\documentclass[11pt]{scrartcl}

\usepackage{authblk}

\usepackage{amsmath,amsthm,amsfonts,amssymb}
\usepackage{mathptmx}
\usepackage{graphics}
\usepackage{color}
\usepackage{cite}
\usepackage{fleqn}
\usepackage{graphicx}
\usepackage{hyperref}
\usepackage{geometry}
\newtheorem{theorem}{Theorem}[section]
\newtheorem{lemma}{Lemma}[section]
\newtheorem{proposition}{Proposition}[section]
\newtheorem{corollary}{Corollary}[section]
\newtheorem{definition}{Definition}[section]
\theoremstyle{plane}
\newtheorem{remark}{Remark}[section]

\newcommand{\e}{^\varepsilon}
\newcommand{\eps}{{\varepsilon}}
\newcommand{\ds}{\displaystyle}
\newcommand{\I}{\mathcal{I}\e}

\renewcommand{\a}{\alpha}
\renewcommand{\b}{\beta}

\newcommand{\cupl}{\bigcup\limits}
\newcommand{\supp}{\mathrm{supp}}
\newcommand{\suml}{\sum\limits}
\newcommand{\intl}{\int\limits}
\newcommand{\liml}{\lim\limits}
\newcommand{\maxl}{\max\limits}
\newcommand{\minl}{\min\limits}
\newcommand{\Om}{{\widetilde{\Omega}}}
\newcommand{\Ga}{{\widetilde{\Gamma}}}

\renewcommand{\phi}{\varphi}

\renewcommand{\d}{\hspace{1pt}\mathrm{d}}

\numberwithin{equation}{section}

\begin{document}

\title{Spectral properties of elliptic operator with double-contrast coefficients near a hyperplane} 

\author{Andrii Khrabustovskyi\thanks{Corresponding author} ,  Michael Plum}
\affil{Institute for Analysis, Department of Mathematics, Karlsruhe Institute of Technology\\
Englerstra{\ss}e 2, 76131 Karlsruhe, Germany; 
Tel.: +49 721 608 42064, Fax.: +49 721 608 46530\\\smallskip
\texttt{andrii.khrabustovskyi@kit.edu, michael.plum@kit.edu}}

\date{\vspace{-5ex}}

\maketitle
\thispagestyle{empty}

\begin{abstract}\noindent
In this paper we study the asymptotic behaviour as $\eps\to 0$ of the spectrum of the
elliptic operator $\mathcal{A}\e=-{1\over b\e}\mathrm{div}(a\e\nabla)$ posed  in a bounded domain $\Omega\subset\mathbb{R}^n$ $(n \geq 2)$ subject to  Dirichlet boundary conditions on $\partial\Omega$.
When $\eps\to 0$ both coefficients $a\e$ and $b\e$ become  high contrast in a small neighborhood of a hyperplane $\Gamma$ intersecting $\Omega$. We prove that the spectrum of $\mathcal{A}\e$ converges to
the spectrum of an operator acting in $L^2(\Omega)\oplus L^2(\Gamma)$ and generated by the operation $-\Delta$ in $\Omega\setminus\Gamma$, the Dirichlet boundary conditions on $\partial\Omega$ and certain interface conditions on $\Gamma$ containing the spectral parameter in a
nonlinear manner. The eigenvalues of this operator may accumulate at a finite point. Then we study the same problem, when $\Omega$ is an infinite straight
strip (``waveguide'') and $\Gamma$ is parallel to its boundary. We  show that $\mathcal{A}\e$ has at least one gap in the spectrum when $\eps$ is small enough and describe the asymptotic behaviour of this gap as $\eps\to 0$. The proofs are based on methods of   homogenization theory.
\end{abstract}

\section{\label{sec1}Introduction}

The problem we are going to study lies on the intersection of spectral theory and  homogenization theory for partial differential operators. We recall that one of the central problems of homogenization theory is to study the asymptotic behaviour as $\eps\to 0$ of the solution $u\e$ to the problem
\begin{gather*}
\mathcal{A}\e u\e=f\text{ in }\Omega,\\
u\e=0\text{ on }\partial\Omega,
\end{gather*}
where $\eps>0$ is a small parameter, $\mathcal{A}\e:= -\ds \mathrm{div}(a\e(x) \nabla )$, 
$\{a\e(x)\}_{\eps>0}$ is a family of real measurable functions satisfying
$$a_-\e\leq a\e(x)\leq a_+\e\text{\ with positive constants }a_\pm\e$$
and becoming highly oscillating as $\eps\to 0$. The typical example is $a\e(x)=\mathbf{a}(x\eps^{-1})$, where $\mathbf{a}(x)$ is a fixed $\mathbb{Z}^n$-periodic function. It is well-known (see, e.g, \cite{CioDon,Allaire}) that if 
\begin{gather}\label{infsup}
\inf\limits_{\eps}a_-\e>0,\quad \sup\limits_{\eps}a_+\e<\infty,
\end{gather}
then the family $\{u\e\}_{\eps}$ is compact in $L^2(\Omega)$, and if $u\e\to u$ as $\eps=\eps_k\to 0$ then
$u(x)$ is a solution of the problem 
\begin{gather*}
\mathcal{A}u=f\text{ in }\Omega,\\
u=0\text{ on }\partial\Omega,
\end{gather*}
where $\mathcal{A}:= -\ds \mathrm{div}(A(x) \nabla )$, $A(x)$ is some bounded matrix-function, bounded away from zero, which depends, in general, on the subsequence $\eps_k$. If $a\e(x)=\mathbf{a}(x\eps^{-1})$ then the whole sequence $u\e$ converges, and in this case $A(x)$ is a constant matrix. 
As we see the limit differential operator has qualitatively the same form as the initial one provided \eqref{infsup} holds.

If, on the contrary, conditions \eqref{infsup} are violated, for example there exist subsets $D\e\subset\Omega$ with non-zero measure such that $\liml_{\eps\to 0}\sup\limits_{x\in D\e}a\e(x)=\infty$ or
\begin{gather}\label{nonlclas}
\lim_{\eps\to 0}\inf_{x\in D\e}a\e(x)=0,
\end{gather} 
then the limit operator may have a more complicated form, which depends essentially on the structure of the domains $D\e$.  
We refer to the monograph \cite{March}, where various problems of this type are studied. In particular, the authors consider  the case, when condition \eqref{nonlclas} holds on the union $D\e$ of thin shells, distributed periodically, with period $\eps$, in the domain $\Omega$. They study the behaviour of linear evolution equations involving  such operators $\mathcal{A}\e$. Semi-linear evolution equations were investigated in 
\cite{Pank0,Pank1,Pank2}. Spectral properties of such operators were studied in \cite{Khrab2}. 

\begin{figure}[h]
\begin{center}
 \begin{picture}(350,190)
 \includegraphics{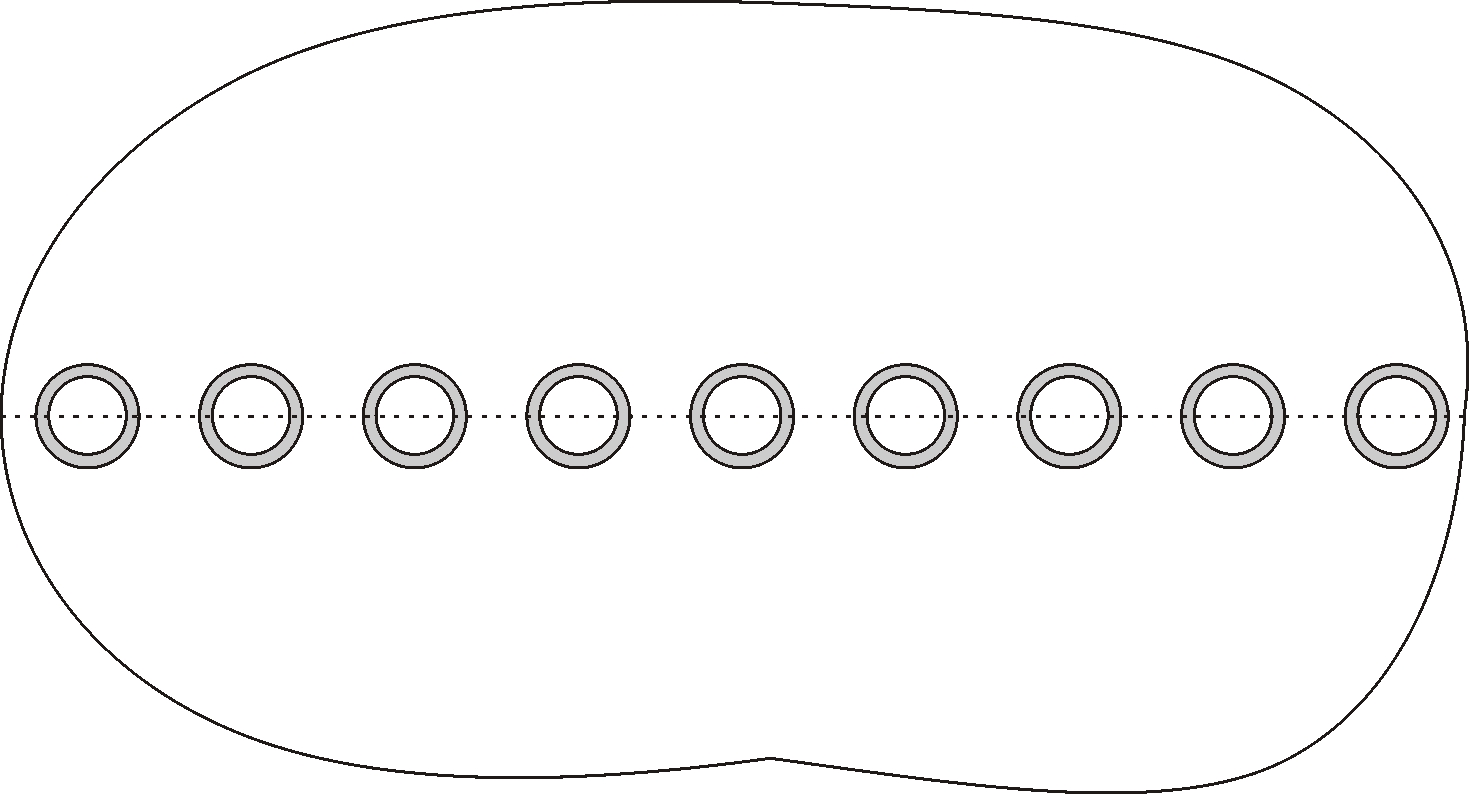}
\put(0, 8){$\Omega$} \put(-4,13){\vector(-3,2){30}}

\put(-70, 115){$\Gamma$} \put(-71,111){\vector(-1,-2){10}}

\put(-280, 116){$D_i\e$} \put(-270,112){\vector(1,-1){11}}
\put(-279,112){\vector(-1,-1){11}}

\put(-320, 60){$B_i\e$} \put(-308,70){\vector(3,4){14}}
\put(-320,70){\vector(-3,4){14}}

\put(-157,64){$^\eps$} \put(-150,70){\vector(1,0){15}}
\put(-161,70){\vector(-1,0){15}}

\put(-165,108){$_{R^\eps=R\eps}$}\put(-175,91){\vector(1,1){9}}
\put(-175,91){\line(1,1){13}}
\put(-162.4,104){\line(1,0){10}}

\put(-205,121){$_{d^\eps=o(\eps)}$}\put(-175,91){\vector(-2,3){5}}
\put(-192,116){\vector(2,-3){10}} \put(-192,116){\line(-1,0){10}}

\put(-178,86){$_{x^{i,\eps}}$}
\put(-175,91){\circle*{2}}

\end{picture}
\caption{}\label{fig1}
\end{center}
\end{figure}

In all papers mentioned above the case of \textit{bulk} distribution of shells was considered. In the present work we are interesting in the case of \textit{surface} distribution of shells, i.e.  the shells are located in a neighbourhood of some hyperplane. 

We notice that our research is inspired, in particular, by spectral problems for periodic differential operators posed in a waveguide-like 
domain. These applications are discussed later in the introduction.

Below we briefly present our main results.
Let $\Omega$ be a bounded domain in $\mathbb{R}^n$ ($n\geq 2$), $\eps>0$ be a small parameter. In $\Omega$ we consider the operator  
\begin{gather}\label{Amain}
\mathcal{A}\e =-\ds{1\over b\e(x)}\mathrm{div}(a\e(x)\nabla)
\end{gather}
subject to Dirichlet boundary conditions on $\partial\Omega$. Here the functions $a\e$ and $b\e$ are bounded above and bounded away from zero uniformly in $\eps$ everywhere except a small neighbourhood of some hyperplane $\Gamma$ having non-empty intersection with the domain $\Omega$.
More precisely, we denote by $D\e=\{D_i\e\}$ a family of thin spherical shells distributed periodically, with period $\eps$,  along $\Gamma$ (counted by the parameter $i$). Each shell has an external radius $R\e= R\eps$ (here $R\in (0,1/2)$ is a constant), the thickness of the shells is $d\e=o(\eps)$ as $\eps\to 0$. 
By $B\e=\{B_i\e\}$ we denote the union of balls surrounded by these shells (see Fig. \ref{fig1}). When $\eps\to 0$ the number $N(\eps)$ of shells goes to infinity as $\eps\to 0$, namely
\begin{gather}\label{total}
N(\eps)\sim { \eps^{1-n} |\Gamma|}
\end{gather}
(hereinafter we will use the same notation $|\cdot|$ for
the volume of a domain in $\mathbb{R}^n$ and as well for
the area of an $(n-1)$-dimensional hypersurface in $\mathbb{R}^{n}$).  We define the functions $a\e(x)$ and $b\e(x)$ as follows: $a\e(x)$ is equal to $1$ for $x\in\Omega\setminus D\e$ and equal to the constant $\a\e>0$ for $x\in D\e$, $b\e(x)$ is equal to $1$ for  $x\in\Omega\setminus B\e$ and equal to the constant $\b\e>0$ for $x\in B\e$.

Operators of the form \eqref{Amain} occur in various areas of physics. For example, in the case $n=3$ the operator $\mathcal{A}\e$ governs the
propagation of acoustic waves in a medium with mass density $({a}\e(x))^{-1}$ and compressibility ${b}\e(x)$. Also,  $\mathcal{A}\e$ describes vibrations of a body occupying the domain $\Omega$, the functions
$a\e(x)$ and $b\e(x)$ are its stiffness and 
mass density, correspondingly.  Notice, that the most interesting effects occur if
$$\a\e=\mathcal{O}(d\e),\quad \beta\e=\mathcal{O}(\eps^{-1})\text{ as }\eps\to 0.$$
In this case $\mathcal{A}\e$ describes vibrations of a body with a lot of tiny  \textit{heavy}  inclusions $B_i\e$ surrounded by thin \textit{soft} layers $D_i\e$. 

Asymptotics of eigenvibrations of a body
with a mass density perturbed near a hypersurface
was studied in a lot of papers -- see, e.g., \cite{Che1,Che2,Lobo1,Lobo2,CCDP,Melnyk2} and references therein.
The case of simultaneously perturbed density and stiffness (``double-contrast'') was studied in \cite{GLNP} (see also \cite{BKS}, where the case of bulk distribution of  double-contrast inclusions was studied). In all these works the geometry of a set supporting the perturbation differs essentially from that one considered in the current paper.

The spectrum  $\sigma(\mathcal{A}\e)$ of $\mathcal{A}\e$ is purely discrete. Our goal is to describe its behaviour as $\eps\to 0$ supposing that the following conditions hold:  
\begin{gather}\label{ass2}
\liml_{\eps\to 0}\ds{(d\e)^2\over\a\e}=0,\\\label{ass3}
\liml_{\eps\to 0}q\e =q\in [0,\infty],\text{ where }q\e=  {\a\e  n\over R d\e \eps\beta\e},\\\label{ass4}
\liml_{\eps\to 0}r\e =r\in [0,\infty),\text{ where }r\e=  R^n\varkappa_n \eps\beta\e.
\end{gather}
Here by $\varkappa_{n}$ we denote the volume of the $n$-dimensional unit ball.
We note, that $q$ is allowed to be infinite. 

Condition \eqref{ass2} means that the eigenfunctions cannot concentrate on the shells (cf.  \eqref{u_inD}). The case $\a\e=\mathcal{O}((d\e)^2)$ differs essentially and will not be studied in this paper.

The parameter $q$ characterizes the ``strength" of the coupling between the domain $\Omega\e:=\Omega\setminus\cupl_i\left(\overline{D_i\e\cup B_i\e}\right)$  and the union of the balls $B_i\e$. We postpone the precise statement to Section \ref{sec3} (see Remark \ref{remark_q}) because first we need to introduce some more notations.

The finiteness of $r$ implies the uniform (with respect to $\eps$) boundedness
of the ``mass'' $M_{B\e}$ of $B\e$, namely, using \eqref{total} and the fact that $d\e=o(\eps)$, we obtain:
\begin{gather*}
M_{B\e}:=\intl_{\cupl_i B_i\e}b\e(x) \d x=\beta\e \suml_{i}|B_i\e|=
\b\e(R\e-d\e)^n \varkappa_n N(\eps)=
r\e {(R\eps-d\e)^n N(\eps)\over R^n \eps }
\sim r|\Gamma|\text{ as }\eps\to 0.
\end{gather*}

In spite of the fact that $\mathcal{A}\e$ contains many parameters the behaviour of its spectrum as $\eps\to 0$ depends essentially only on $q$ being finite or infinite and $r$ being positive or zero. 

At this point we present the results in a formal way, more precise statements are formulated in the next section using the language of operator theory. Below the convergence of spectra is understood in the Hausdorff sense, see Definition \ref{def1}.
One has:
\begin{itemize}
\item Let $q<\infty$. In this case $\sigma(\mathcal{A}\e)$ converges, as $\eps\to 0$, to the union of the point $q$ and the set of eigenvalues of the $\lambda$-nonlinear spectral problem
\begin{gather}\label{formal1}
\begin{cases}
-\Delta u=\lambda u&\text{ in }\Omega\setminus\Gamma,\\ 
\left[u\right]=0,\quad \left[{\partial u\over\partial  n}\right]={\lambda qr\over q-\lambda} u&\text{ on }\Gamma,\\
u=0&\text{ on }\partial\Omega,
\end{cases}
\end{gather}
where the brackets denote the jump of the enclosed quantities.  

If $r>0$ then the set of eigenvalues of the problem \eqref{formal1} consists of two ascending sequences --- one of them tends to infinity and the second one tends to $q$ (in the case $q=0$ the second sequence is not present). 

We notice that in the case $\Gamma\subset\partial\Omega$ the interface conditions in \eqref{formal1} reduce formally to the boundary conditions
\begin{gather}\label{meromo}
{\partial u\over\partial  n}=\mathcal{F}(\lambda) u,
\end{gather}
where $n$ in the outward pointing unit normal to $\partial\Omega$, $\mathcal{F}(\lambda)={\lambda qr\over q-\lambda}$.
The boundary conditions of the form \eqref{meromo} 
appear in limit boundary value problems for various homogenization problems -- see, e.g., \cite{Lobo1,MN} (here $\mathcal{F}$ has infinitely many poles), \cite{CK} (here $\mathcal{F}$ has one pole).

\item Let $q=\infty$. In this case $\sigma(\mathcal{A}\e)$ converges to the set of  eigenvalues of the following Steklov-type  spectral problem:
\begin{gather}\label{formal2}
\begin{cases}
-\Delta u=\lambda u&\text{ in }\Omega\setminus\Gamma,\\ 
\left[u\right]=0,\quad \left[{\partial u\over\partial  n}\right]=\lambda r u&\text{ on }\Gamma,\\
u=0&\text{ on }\partial\Omega.
\end{cases}
\end{gather}
Note, that \eqref{formal2} is a formal limit of  \eqref{formal1} as $q\to\infty$.

\item The following estimate is valid:
\begin{gather}\label{tr}
\sup\limits_{k\in\mathbb{N}}\left(\underset{\eps\to 0}{\overline{\lim}}\lambda_k\e\right)\leq q,
\end{gather}
where $\lambda_k\e$ is the $k$-th eigenvalue of $\mathcal{A}\e$. 
From \eqref{tr} we immediately obtain
$$\forall k:\ \lambda_k\e\to 0\text{ as }\eps\to 0\text{ provided }q=0.$$
\end{itemize}

Note, that in the case $r=0$ the problems \eqref{formal1} and \eqref{formal2} reduce to the eigenvalue problem for the Dirichlet Laplacian in $\Omega$.

Also we note, that the choice of the boundary conditions on $\partial\Omega$ is not essential -- instead of the Dirichlet boundary conditions we can impose, for example, Neumann or mixed ones. These conditions will be inherited by the limit spectral problem.
\vspace{2mm}

In the last part of the paper we consider the same problem for a waveguide-like domain $\Omega$:
$$\Omega=\mathbb{R}\times \left(d_-,d_+\right),\quad \Gamma=\{x=(x_1,0):\ x_1\in\mathbb{R}\}.$$
where $d_-<0$, $d_+>0$. Here we are interested in the case $q>0$, $r>0$ only.

Due to the periodicity of $\mathcal{A}\e$ its spectrum is a locally finite union of compact intervals (bands). In general the bands may overlap and the natural question arising here
is whether gaps open up in the spectrum (i.e. whether there is an open interval $(\a,\b)\subset (0,\infty)$ such that $(\a,\b)\cap \sigma(\mathcal{A}\e)=\varnothing$, while $\a,\b\in \sigma(\mathcal{A}\e)$). This problem is interesting for applications since the presence of gaps is important for the description of wave processes which are governed by the differential operators under consideration. Namely, if the wave frequency belongs to a gap, then the corresponding wave cannot propagate in the medium without attenuation. This feature is a dominant requirement for so-called photonic crystals which are materials with periodic dielectric structure attracting much attention in recent years (see, e.g., \cite{Dorfler,Kuchment_PC}).

It was proved in \cite{Khrab2} that in the case $\Omega=\mathbb{R}^n$ and a bulk distribution of shells, the spectrum of $\mathcal{A}\e$ has a gap when $\eps$ is small enough. 

Our goal is to study whether  gaps will open up in case of our waveguide-like domain. We will prove that the spectrum of $\mathcal{A}\e$ converges in the Hausdorff sense to the spectrum of some  operator $\mathcal{A}$ which is 
 defined by the same differential expression as in the case of a bounded domain and its spectrum is as follows: if $q<\min\left\{{\pi^2\over d_-^2},{\pi^2\over d_+^2}\right\}$ then
$$\sigma(\mathcal{A})=
[\a_1,q]\cup[\a_2,\infty),$$
otherwise $\sigma(\mathcal{A})=[\a_1,\infty)$. Here $\a_1,\a_2$ are some positive numbers satisfying
$0<\a_1<q<\a_2$.
Thus if the waveguide is thin enough we have a gap in the spectrum of 
$\mathcal{A}$ (and, consequently, $\sigma(\mathcal{A}\e)$ has a gap when $\eps$ is small enough). 

Periodic perturbations of the Laplacian in wavegide-like domains leading to opening of spectral gaps were also studied in  \cite{ExKh,Borisov1,Borisov2,Nazarov1,Nazarov2-,Nazarov2,Nazarov3,Nazarov4,Yoshi}. In all these papers (except \cite{Borisov2,ExKh}) spectral gaps appear because of a perturbation of the boundary of the waveguide (for example by making small holes periodically distributed along the waveguide \cite{Nazarov2} or by dividing the waveguide into two parts coupled by a periodic system of small windows \cite{Borisov1}). In the  paper \cite{Borisov2} the authors
considered small perturbations of the Laplace operator in a cylindrical domain by second-order differential operators with periodic coefficients; they gave sufficient conditions on this perturbation for gap opening. These conditions are not valid for the operators considered in the present work.
In the paper \cite{ExKh} the authors perturbed the Laplace operator by a singular potential supported by a family of periodically distributed surfaces.
\medskip

The paper is organized as follows. In Section \ref{sec2} we set the problem  
and formulate  the main result (Theorem \ref{th1}), also we prove the estimate \eqref{tr} (Theorem \ref{th2}). Theorem \ref{th1} will be proven in Section \ref{sec3}: 
the case $q<\infty$ in Subsection \ref{subsec32} and the case $q=\infty$ in Subsection \ref{subsec34}. Finally, in Section \ref{sec4} we consider the case of a waveguide.

\section{\label{sec2} Setting of the problem and main result}

\subsection{\label{subsec21} The operator $\mathcal{A}\e$}

In what follows we denote the Cartesian coordinates in $\mathbb{R}^n$ by $x=(x_1,\dots,x_n)$ . 
Let $\Omega\subset \mathbb{R}^n$ be a bounded Lipschitz domain ($n\geq 2$). It is supposed that $0\in\Omega$. We denote by $\Gamma$ the intersection of $\Omega$ with the hyperplane $\{x_n=0\}$:
$$\Gamma=\left\{x\in\Omega:\ x_n=0\right\}.$$

Let $\eps>0$ be a small parameter. We denote by $x^{i,\eps}$, $i=(i_1,\dots,i_n)\in \mathbb{Z}^{n-1}$ the family of points, distributed periodically, with period $\eps$, on the plane $\left\{x_n=0\right\}$:
$$x^{i,\eps}=(\eps i_1,\eps i_2,\dots,\eps i_{n-1},0)$$
and introduce the following sets (see Fig. \ref{fig1}):
\begin{gather*}
D_{i}\e=\left\{x\in\mathbb{R}^n:\ R\e-d\e<|x-x^{i,\eps}|<R\e\right\},\quad
B_{i}\e=\left\{x\in\mathbb{R}^n:\ |x-x^{i,\eps}|<R\e-d\e\right\}.
\end{gather*}
Here
$$R\e=R\eps,\text{ where }R\in \left(0,{1\over 2}\right),\quad d\e=o(\eps)\text{ as }\eps\to 0.$$

We denote by  $\mathcal{I}\e$ the set of all multiindices $i\in\mathbb{Z}^{n-1}$ satisfying 
\begin{gather}\label{xingamma}
x^{i,\eps}\in\Gamma,\quad \mathrm{dist}(x^{i,\eps},\partial\Omega)\geq \eps{\sqrt{n}\over 2}.
\end{gather}
The inequality in \eqref{xingamma} implies that
the cube with center at $x^{i,\eps}$, side length $\eps$,  being oriented along the coordinate axes, belongs to $\overline{\Omega}$ whenever $i\in\mathcal{J}\e$. This condition is technical and is needed only to simplify the proof presentation.

We introduce the piecewise constant functions
\begin{gather}\label{ab}
a\e(x)=
\begin{cases}
1,&x\in\Omega\setminus \cupl_{i\in\mathcal{I}\e}D_i\e,\\
\alpha\e,&x\in \cupl_{i\in\mathcal{I}\e}D_i\e,
\end{cases}\quad b\e(x)=
\begin{cases}
1,&x\in\Omega\setminus \cupl_{i\in\mathcal{I}\e}B_i\e,\\
\beta\e,&x\in \cupl_{i\in\mathcal{I}\e}B_i\e,
\end{cases}
\end{gather}
where $\a\e,\ \b\e$ are positive constants.

Now, let us define accurately the operator formally given by $$\mathcal{A}\e=-{1\over b\e }\mathrm{div}(a\e\nabla)$$ 
(subject to the Dirichlet boundary conditions on $\partial\Omega$).
By $\mathcal{H}\e$ we denote the Hilbert space of functions from $L^2(\Omega)$ endowed with a scalar product
\begin{gather}\label{He}
(u,v)_{\mathcal{H}\e}=\intl_\Omega u(x)\overline{v(x)}b\e(x)  \d x.
\end{gather}
By $\eta\e$ we denote the sesquilinear form in $\mathcal{H}\e$ defined by  
\begin{gather}\label{eta}
\eta\e[u,v]=\intl_\Omega a\e(x)\nabla u\cdot\nabla\bar v \d x
\end{gather}
with $\mathrm{dom}(\eta\e)=H^1_0(\Omega)$. The form  $\eta\e$ is densely defined, closed, positive and symmetric, whence (cf. \cite[Chapter 6, Theorem 2.1]{K66}) there
exists a unique self-adjoint and positive operator
$\mathcal{A}\e$ associated with the form
$\eta\e$, i.e.
\begin{gather*}
(\mathcal{A}\e u,v)_{\mathcal{H}\e}=
\eta\e[u,v],\quad\forall u\in
\mathrm{dom}(\mathcal{A}\e),\ \forall v\in
\mathrm{dom}(\eta\e).
\end{gather*}

The domain of  $\mathcal{A}\e$ consists of all functions $u\in H^1_0(\Omega)$ with the corresponding restrictions
belonging to the spaces $H^2(D_i\e)$,
$H^2(B_i\e)$ (for all $i\in\mathcal{I}\e$), $H^2\left(\Omega\setminus\cupl_{i\in\I }\left(\overline{ D_i\e\cup B_i\e}\right)\right)$, and
satisfying the following conditions on $\partial D_{i}\e$:
\begin{gather}\label{conjugation}
\begin{cases} (u)^+=(u)^-\text{\quad and\quad } \ds\left({\partial
u\over\partial  n}\right)^+=\a\e\left({\partial u\over\partial
n}\right)^-,
&x\in\partial D_{i}\e\setminus\partial B_i\e,\\
(u)^+=(u)^-\text{\quad and\quad }\a\e\ds\left({\partial
u\over\partial  n}\right)^+=\left({\partial u\over\partial
n}\right)^-, &x\in \partial B_i\e
\end{cases}
\end{gather}
where $+$ (respectively, $-$) denote the traces of the
function $u$ and its normal derivative taken from the exterior
(respectively, interior) side of either
$\partial D_{i}\e\setminus\partial B_i\e$ or
$\partial B_i\e$. 

The spectrum $\sigma(\mathcal{A}\e)$ of the operator $\mathcal{A}\e$ is purely discrete. Our goal is to describe the behaviour of $\sigma(\mathcal{A}\e)$ as $\eps\to 0$ under the assumption that conditions \eqref{ass2}-\eqref{ass4} hold.

\subsection{\label{subsec22} The limit operator}

Next we introduce the limit operator $\mathcal{A}_{q,r}$.

Let $r>0$. We denote by $\mathcal{H}_r$ the Hilbert space of functions from $L^2(\Omega)\oplus L^2(\Gamma)$ endowed with the scalar product
\begin{gather}\label{H_qr}
(U,V)_{\mathcal{H}_r}=\intl_{\Omega}u_1(x)\overline{v_1(x)} \d x+r\intl_\Gamma u_2(x)\overline{v_2(x)}  \d s,\quad U=\left(u_1,u_2\right),\ V=(v_1,v_2).
\end{gather}
Hereinafter, we use the standard notation $\d s$ for the density of the measure generated 
on $\Gamma$ (or any other $(n-1)$-dimensional hypersurface) by the Euclidean metric in $\mathbb{R}^n$.

For $q<\infty$ we introduce the sesquilinear form $\eta_{q,r}$ in $\mathcal{H}_r$ by the formula
\begin{gather}\label{eta_qr}
\eta_{q,r}[U,V]=\intl_\Omega {\nabla u_1}\cdot {\nabla \overline{v_1}}\d x+qr\intl_\Gamma (u_1-u_2)(\overline{v_1-v_2})\d s,\quad U=\left(u_1,u_2\right),\ V=(v_1,v_2)
\end{gather}
with $\mathrm{dom}(\eta_{q,r})=H_0^1(\Omega)\oplus L^2(\Gamma)$. 
Here we use the same notation for the functions $u_1$, $v_1$ and their traces on $\Gamma$.

For $q=\infty$ we introduce the sesquilinear form $\eta_{\infty,r}$ in $\mathcal{H}_r$ by the formula
\begin{gather*}
\eta_{\infty,r}[U,V]=\intl_\Omega {\nabla u_1}\cdot {\nabla \overline{v_1}}\d x,\ U=(u_1,u_2),\ V=(v_1,v_2)
\end{gather*}
with $\mathrm{dom}(\eta_{\infty,r})=\left\{U=(u_1,u_2)\in H_0^1(\Omega)\oplus L^2(\Gamma):\ u_1|_\Gamma=u_2\text{ on }\Gamma\right\}$.

The forms $\eta_{q,r}$ ($q<\infty$) and $\eta_{\infty,r}$ are densely defined, closed, positive and symmetric. Then we defined the limit operator $\mathcal{A}_{q,r}$ by
$$\mathcal{A}_{q,r}=
\begin{cases}
\text{the operator acting in $\mathcal{H}_r$ and associated with the form  }\eta_{q,r},&q<\infty,\ r>0,\\
\text{the operator acting in $\mathcal{H}_r$ and associated with the form  }\eta_{\infty,r},&q=\infty,\ r>0,\\
(-\Delta_{\Omega})\oplus qI,\text{ acting in }L^2(\Omega)\oplus L^2(\Gamma),&q<\infty,\ r=0,\\
(-\Delta_{\Omega}),\text{ acting in }L^2(\Omega),&q=\infty,\ r=0,\\
\end{cases}$$
where $\Delta_\Omega$ is the Dirichlet Laplacian in $\Omega$ with domain $H^2(\Omega)\cap H^1_0(\Omega)$, $I$ is the identity operator.

\subsection{\label{subsec23} Spectrum of the operator $\mathcal{A}_{q,r}$}

Let $q<\infty$, $r>0$. If $q=0$ then $\mathcal{A}_{q,r}$  is a direct sum of the operator $-\Delta_{\Omega}$ and the null operator in $L^2(\Gamma)$, implying 
$\sigma(\mathcal{A}_{0,r})=\sigma(-\Delta_{\Omega})\cup\{0\}.$
In the case $q>0$ one has the following lemma.

\begin{lemma}\label{lmAqr}
Let $q>0$. Then the spectrum of the operator $\mathcal{A}_{q,r}$ has the form
\begin{gather*}
\label{spectrum} \sigma(\mathcal{A}_{q,r})=\{q\}\cup \left(\cupl_{k\in\mathbb{N}}\{\lambda_k^-\}\right)\cup
\left(\cupl_{k\in\mathbb{N}}\{\lambda_k^+\}\right).
\end{gather*}
The points $\lambda_k^{\pm}, k\in\mathbb{N}$ belong to the discrete
spectrum, $q$ is a point of the essential spectrum and 
\begin{gather}\label{distr}
0<\lambda_1^{+}\leq \lambda_2^{+}\leq ...\leq\lambda_k^{+}\leq\dots\underset{k\to\infty}\to
q<  \lambda_1^{-}\leq
\lambda_2^{-}\leq ...\leq\lambda_k^{-}\leq\dots\underset{k\to\infty}\to \infty.
\end{gather}
\end{lemma}
We will prove this lemma in Subsection \ref{subsec33}.\smallskip

It is easy to see that the operator $\mathcal{A}_{\infty,r}$ ($r>0$) has compact resolvent in view of the trace theorem and the Rellich embedding theorem. Therefore the spectrum of  $\mathcal{A}_{\infty,r}$ is purely discrete. 
\smallskip

Finally,
$$\sigma(\mathcal{A}_{q,0})=\sigma(-\Delta_{\Omega})\cup\{q\}.$$

\begin{remark}
Let $r>0$. It is easy to see that

\begin{itemize}
\item[-] if $q<\infty$ then 
\begin{multline*}
\lambda\text{ is an eigenvalue of }\mathcal{A}_{q,r},\ U=(u_1,u_2)\text{ is the corresponding eigenfunction}\\ \Longleftrightarrow\  u_2={q  u_1|_\Gamma\over q-\lambda}\text{ and }(\lambda,u_1)\text{ is an eigenpair of \eqref{formal1}}.
\end{multline*}

\item[-] if $q=\infty$ then 
\begin{multline*}
\lambda\text{ is an eigenvalue of }\mathcal{A}_{\infty,r},\ U=(u_1,u_2)\text{ is the corresponding eigenfunction}\\ \Longleftrightarrow\  u_2={u_1|_\Gamma}\text{ and }  (\lambda,u_1)\text{ is an eigenpair of \eqref{formal2}}.
\end{multline*}

\end{itemize}
Notice, that problem  \eqref{formal2} is a formal limit of  problem \eqref{formal1} as $q\to \infty$; that justifies the notation $\mathcal{A}_{\infty,r}$.
Also, setting $r=0$ in  \eqref{formal1} or \eqref{formal2},
we arrive at the eigenvalue problem for the Dirichlet Laplacian in $\Omega$; that justifies the notation $\mathcal{A}_{q,0}$.
\end{remark}

\subsection{\label{subsec24} The main results}

In what follows speaking about the convergence of spectra we will use the 
concept of  \textit{Hausdorff convergence}. 
 
\begin{definition}\label{def1}
The Hausdorff distance between two compact sets $X,Y\subset\mathbb{R}$ 
is defined as follows:
$$\mathrm{dist}_H(X,Y):=\max\limits\left\{\sup\limits_{
x\in X }\inf\limits_{y\in Y} |x-y|;
\sup\limits_{
y\in Y }\inf\limits_{x\in X} |y-x|\right\}.$$
The sequence of compact sets $X\e\subset\mathbb{R}$ converges to the compact set 
$X\subset\mathbb{R}$ in the Hausdorff sense if 
$$\mathrm{dist}_H(X\e,X)\to 0\text{ as }\eps\to 0.$$
\end{definition}

Now, we are in position to formulate the main result of this paper. 

\begin{theorem}\label{th1} 
Let $l\subset \mathbb{R}$ be an arbitrary compact interval.
Then the set $\sigma(\mathcal{A}\e)\cap l$ converges in the Hausdorff 
sense as $\eps\to 0$ to the set $\sigma(\mathcal{A}_{q,r})\cap l$.
\end{theorem}

\begin{remark}\label{remark2}
It is straightforward to show that the claim of Theorem \ref{th1} is equivalent
to the following two properties:
\begin{gather}\tag{A}\label{ha}
\text{if }\lambda\e\in\sigma(\mathcal{A}\e)\text{ and }\liml_{\eps\to
0}\lambda\e=\lambda\text{ then }\lambda\in
\sigma(\mathcal{A}_{q,r}),\\\tag{B}\label{hb} \text{for any }\lambda\in
\sigma(\mathcal{A}_{q,r})\text{ there exist }\lambda\e\in\sigma(\mathcal{A}\e)\text{ such
that }\liml_{\eps\to 0}\lambda\e=\lambda.
\end{gather}
\end{remark}

Before starting the proof of Theorem \ref{th1} we obtain an estimate concerning the 
behaviour of the $k$-th eigenvalue of $\mathcal{A}\e$. Note, that in the case $q>0$, $r>0$ this estimate follows easily from Theorem \ref{th1} and Lemma \ref{lmAqr}.

\begin{theorem}\label{th2}One has
\begin{gather}\label{qqq}
\sup\limits_{k\in\mathbb{N}}\left(\underset{\eps\to 0}{\overline{\lim}}\lambda_k\e\right)\leq q,
\end{gather}
where $\{\lambda_k^\eps\}_{k\in\mathbb{N}}$ is the
sequence of eigenvalues of the operator $\mathcal{A}\e$ written in ascending order and repeated according to multiplicity.
\end{theorem}

\begin{proof}Below we assume that $q<\infty$, otherwise the theorem is trivial.
One has the following min-max principle (cf. \cite[\S 4.5]{Davies}):
\begin{gather}\label{minmax+}
\lambda_k\e=\inf\limits_{L\in \mathcal{L}_k}\left(\sup\limits_{0\not= v\in L}{\eta\e[v,v]\over \|v\|^2_{\mathcal{H}\e}}\right).
\end{gather}
Here $\mathcal{L}_k$ is the set of all $k$-dimensional subspaces of $\mathrm{dom}(\eta\e)=H_0^1(\Omega)$.

We fix $k$ arbitrary pairwise different indices $\mathbf{i}\e_j\in\I$, $j=1,\dots, k$ and 
introduce the following continuous and piecewise smooth functions:
\begin{gather*}
v_j\e(x)=\begin{cases}
1,& x\in B\e_{\mathbf{i}\e_j},\\
\ds{G(|x-x^{i,\eps}|)},& x\in D\e_{\mathbf{i}\e_j},\\
0,&\text{otherwise}.
\end{cases}
\end{gather*}
where the function $G:\mathbb{R}\to\mathbb{R}$ is defined by
$$G(\rho)=\begin{cases}
\ds{\rho^{2-n}-(R\e)^{2-n}\over (R\e-d\e)^{2-n}-(R\e)^{2-n}},&n>2,\\
\ds{\ln\rho-\ln R\e\over \ln(R\e-d\e)-\ln R\e},&n=2.
\end{cases}$$

We denote 
$$L':=\mathrm{span}\{v_j\e,\ j=1,\dots,k\}.$$
Obviously, $L'\subset H_0^1(\Omega)$. Moreover, since $\mathrm{supp}(v_{j_1}\e)\cap\mathrm{supp}(v_{j_2}\e)=\varnothing$ as $j_1\not= j_2$, we conclude that  $\mathrm{dim}(L')=k$. Therefore $L'\in\mathcal{L}_k$. 

Taking into account that $d\e=o(\eps)$ we obtain the following asymptotics as $\eps\to 0$:
\begin{gather*}
\eta\e[v_j\e,v_j\e]=\left.\begin{cases}\a\e(n-2)\left((R\e-d\e)^{2-n}-(R\e)^{2-n}\right)^{-1}\omega_{n-1},&n>2\\
\a\e\left(\ln R\e-\ln(R\e-d\e)\right)^{-1},&n=2
\end{cases}\right\}\sim  {\a\e  R^{n-1} \omega_{n-1} \eps^{n-1}\over  d\e},
\\
\| v_j\e\|^2_{\mathcal{H}\e}=
\b\e R^n \eps^n\varkappa_n + \mathcal{O}(d\e \eps^{n-1}),
\end{gather*}
where $\omega_{n-1}$ is the area of the $(n-1)$-dimensional unit sphere.
Taking into account that $\omega_{n-1}=n\varkappa_n$ and using \eqref{ass2}, we obtain easily:
\begin{gather}\label{frac}
{\eta\e[v_j\e,v_j\e]\over \|v_j\e\|^2_{\mathcal{H}\e}}
={\a\e n\over Rd\e\eps\b\e\left(1+\mathcal{O}\left({d\e\over\eps\b\e}\right)\right)}
={\a\e n\over Rd\e\eps\b\e\left(1+q\e {(d\e)^2\over\a\e}\mathcal{O}(1)\right)}\sim q\text{ as }\eps\to 0. 
\end{gather}
Since the supports of $v_j\e$ are pairwise disjoint, \eqref{frac} holds true  for any arbitrary $v\in {L}'$ instead of $v_j\e$. 

Finally, using \eqref{minmax+} and \eqref{frac}, we obtain
\begin{gather*}
\lambda_k\e\leq \sup\limits_{v\in L'}{\eta[v,v]^2_{L^2(\Omega)}\over \|v\|^2_{\mathcal{H}\e}}\sim q. 
\end{gather*}
The theorem is proved.
\end{proof}

\begin{corollary}
For each $k\in\mathbb{N}$ $\lambda_k\e\to 0\text{ as }\eps\to 0$ provided $q=0$.
\end{corollary}

\section{\label{sec3}Proof of Theorem \ref{th1}}

We present the proof for the case $n\geq 3$ only. For the case $n=2$ the proof needs some small modifications (for example in \eqref{v} the function $|x-x^{i,\eps}|^{2-n}$ has to be replaced by $-\ln|x-x^{i,\eps}|$).

\subsection{\label{subsec31}Preliminaries}

In what follows by $C,C_1,C_2...$ we denote generic constants that do
not depend on $\eps$.

By $\langle u \rangle_B$ we denote the mean value of the function
$u(x)$ over the domain $B$: $$\langle u \rangle_B={1\over
|B|}\intl_B u(x)\d x.$$
If $\Sigma\subset \mathbb{R}^n$ is an $(n-1)$-dimensional surface then 
the Euclidean metric in $\mathbb{R}^n$ induces on $\Sigma$
the Riemannian metric and measure. As before we denote by $\d s$ the density
of this measure, and by $\langle u\rangle_\Sigma$ the
mean value of the function $u$ over $\Sigma$, i.e $\langle
u\rangle_\Sigma={1\over |\Sigma|}\intl_{\Sigma}u \d s$, $|\Sigma|=\intl_{\Sigma}\d s$.

We introduce the following sets (the sets $D_i\e$, $B_i\e$ were introduced above in Section \ref{sec2}):
\begin{itemize}

\item $\Omega\e=\Omega\setminus\cupl_{i\in\I }\left(\overline{ D_i\e\cup B_i\e}\right)$,

\item $S_i^{\eps,+}=\left\{x\in\mathbb{R}^n:\ |x-x^{i,\eps}|=R\e\right\},$

\item $S_i^{\eps,-}=\left\{x\in\mathbb{R}^n:\ |x-x^{i,\eps}|=R\e-d\e\right\},$

\item $Y_i\e=\left\{x=(x_1,\dots,x_n)\in\mathbb{R}^n:\ |x_k-(x^{i,\eps})_k|<{\eps\over 2},\ k=1,\dots ,n\right\}$, where $(x^{i,\eps})_k$ is the $k$-th coordinate of $x^{i,\eps}$,

\item $\Gamma_i\e=Y_i\e\cap\Gamma.$

\end{itemize}

One has
\begin{gather}\label{disttogamma1}\cupl_{i\in\I}Y_i\e\subset\Omega,\\
\label{disttogamma2}
\cupl_{i\in\I}\Gamma_i\e\subset\Gamma,\quad \liml_{\eps\to 0}\left|\Gamma\setminus \cupl_{i\in\I}\Gamma_i\e\right|=0
\end{gather}
(recall that  the set $\I$ consists of $i\in \mathbb{Z}^{n-1}$  satisfying  $x^{i,\eps}\in\Gamma$ and $\mathrm{dist}(x^{i,\eps},\partial\Omega\setminus\Gamma)\geq \eps {\sqrt{n}\over 2}$, whence one can easily obtain \eqref{disttogamma1}-\eqref{disttogamma2}). 

By $u_1\e,u_2\e,\dots,u_k\e\dots$ we denote a sequence of eigenfunctions of $\mathcal{A}\e$ 
corresponding to the non-decreasing sequence 
$\{\lambda\e_k\}_{k\in\mathbb{N}}$ of eigenvaluesб
and normalized by the condition
$$(u_k\e,u_l\e)_{\mathcal{H}\e}=\delta_{kl}.$$

The following lemmata will be frequently used throughout the proof.

\begin{lemma}\label{lemma_est}
One has the following estimates: $\forall u\in H^1(Y_i\e)$
\begin{gather}\label{aa1}
\left|\langle u\rangle_{S_i^{\eps,+}}-\langle u\rangle_{\Gamma_i\e}\right|^2\leq C\eps^{2-n}\|\nabla u\|^2_{Y_i\e},\\\label{aa2}
\left|\langle u\rangle_{S_i^{\eps,-}}-\langle u\rangle_{B_i\e}\right|^2\leq C\eps^{2-n}\|\nabla u\|^2_{B_i\e}.
\end{gather}
\end{lemma}

\begin{proof} 
Using a standard trace inequality and rescaling arguments one can easily obtain the following inequality:
\begin{gather}\label{tr1}
\|v\|^2_{L^2(S_i^{\eps,-})}\leq C\left(\eps^{-1} \|v\|^2_{L^2(B_i^{\eps})}+\eps \|\nabla v\|^2_{L^2(B_i^{\eps})}\right),\ \forall v\in H^1(B_i\e).
\end{gather}

We set $v:=u-\langle u\rangle_{B_i\e}$. Using \eqref{tr1}, the Cauchy-Schwarz inequality 
and the Poincar\'e inequality  
$$\|u-\langle u\rangle_{B_i\e}\|^2_{L^2(B_i^{\eps})}\leq C\eps^{2}\|\nabla u\|^2_{L^2(B_i^{\eps})},$$
we obtain:
\begin{multline*}
\left|\langle u\rangle_{S_i^{\eps,-}}-\langle u\rangle_{B_i\e}\right|^2=\left|\langle v\rangle_{S_i^{\eps,-}}\right|^2\leq {1\over |S_i^{\eps,-}|}\|v\|^2_{L^2(S_i^{\eps,-})}\\\leq
C\left(\eps^{-n} \|u-\langle u\rangle_{B_i\e}\|^2_{L^2(B_i^{\eps})}+\eps^{2-n} \|\nabla u\|^2_{L^2(B_i^{\eps})}\right)\leq C_1\eps^{2-n}\|\nabla u\|^2_{L^2(B_i^{\eps})}
\end{multline*}
and \eqref{aa2} is proved.

In the same way we prove the estimates
\begin{gather}\notag
\left|\langle u\rangle_{S_i^{\eps,+}}-\langle u\rangle_{Y_i\e}\right|^2\leq C\eps^{2-n}\|\nabla u\|^2_{Y_i\e},\\\label{GammaY}
\left|\langle u\rangle_{\Gamma_i^{\eps}}-\langle u\rangle_{Y_i\e}\right|^2\leq C\eps^{2-n}\|\nabla u\|^2_{Y_i\e},
\end{gather}
which gives \eqref{aa1}. The lemma is proved.
\end{proof}

\begin{lemma}\label{lemma_pm}
One has the following inequality:  $\forall u\in H^1(D_i\e)$ 
\begin{gather}\label{ineq_pm}
\left|\langle u\rangle_{S_i^{\eps,+}}-\langle u\rangle_{S_i^{\eps,-}}\right|^2\leq C d\e\eps^{1-n}\|\nabla u\|_{L^2(D_i\e)}^2.
\end{gather}
\end{lemma}

\begin{proof}
By   density arguments it is enough to prove this lemma only for smooth functions.

We introduce in $D_{i}\e$ the spherical coordinates $(\rho,\Theta)$,
where $\rho\in (R\e-d\e,R\e)$ is the distance to $x^{i,\eps}$, $\Theta$ are the angle
coordinates. By $\mathrm{S}_{n-1}$ we denote the
$(n-1)$-dimensional unit sphere, by $\d\Theta$ we denote the
Riemannian measure on $\mathrm{S}_{n-1}$. One has
$$u(R\e,\Theta)-u(R\e-d\e,\Theta)=\intl_{R\e-d\e}^{R\e} {\partial
u\over\partial \rho}(\rho,\Theta)\d \rho.$$ We integrate this equality over $\mathrm{S}_{n-1}$ (with respect to $\Theta$),
divide by $|\mathrm{S}_{n-1}|$ and square. Using the Cauchy-Schwarz inequality and taking into account that $d\e=o(\eps)$, $R\e=R\eps$, we
obtain
\begin{multline*}
\left|\langle u\rangle_{S_{i}^{\eps,+}}-\langle u\rangle_{S_{i}^{\eps,-}}\right|^2=
\left|{1\over |\mathrm{S}_{n-1}|}\intl_{\mathrm{S}_{n-1}}
\intl_{R\e-d\e}^{R\e} {\partial
u\over\partial \rho}(\rho,\Theta)\d \rho \d\Theta\right|^2\\\leq
C \left(\intl_{\mathrm{S}_{n-1}}\intl_{R\e-d\e}^{R\e}\left|
{\partial u\over\partial \rho}(\rho,\Theta)\right|^2
\rho^{n-1}\d \rho \d\Theta\right)\cdot\left(
\intl_{R\e-d\e}^{R\e} {\d \rho\over \rho^{n-1}}\right)\\\leq
C_1\ds\left({1\over (R\e-d\e)^{n-2}}-{1\over (R\e)^{n-2}}\right)\|\nabla u\|^2_{L^2( D_{i}\e)}\leq 
C_2{d\e\eps^{1-n}}\|\nabla u\e\|^2_{L^2(D_i\e)}.
\end{multline*}
The lemma is proved.
\end{proof}

\subsection{\label{subsec32}Proof of Theorem \ref{th1}: the case $q<\infty$}

Recall that  the claim of Theorem \ref{th1} is equivalent
to the fulfilment of properties \eqref{ha}-\eqref{hb} (see Remark \ref{remark2}).

\subsubsection{Proof of property \eqref{ha}} 

Let $\lambda\e\in\sigma(\mathcal{A}\e)$ and  $\lambda\e\to\lambda$ as $\eps\to 0$. We have to prove that $\lambda\in\sigma(\mathcal{A}_{q,r})$.

We denote by $k(\eps)$ the index corresponding to $\lambda\e$ (i.e., $\lambda\e=\lambda_{k(\eps)}\e$).
By $u\e=u\e_{k(\eps)}\in H^1_0(\Omega)$ we denote the corresponding eigenfunction. One has
\begin{gather}\label{normalization1}
1=\|u\e\|_{\mathcal{H}\e}=
\|u\e\|^2_{L^2(\Omega\e)}+
\suml_{i\in\mathcal{I}\e}\|u\e\|^2_{L^2(D_i\e)}+
\b\e \suml_{i\in\mathcal{I}\e}\|\nabla u\e\|^2_{L^2(B_i\e)},
\\\label{normalization2} \lambda\e=\eta\e[u\e,u\e]=
\|\nabla u\e\|^2_{L^2(\Omega\e)}+
\a\e  \suml_{i\in\mathcal{I}\e}\|\nabla u\e\|^2_{L^2(D_i\e)}+
 \suml_{i\in\mathcal{I}\e}\| u\e\|^2_{L^2(B_i\e)}.
\end{gather}

In order to describe the behavior of $u\e$ as $\eps\to 0$ we need some additional operators.
It is known (see, e.g., \cite[Chapter 4, \S 4.2]{March}) that there exists an extension operator $\Pi_1\e:H^1(\Omega\e)\to H^1(\Omega)$ uniformly bounded with respect to $\eps$, i.e. the following conditions hold:
\begin{gather}\label{Pi1ineq}
\begin{array}{l}
[\Pi_1\e u](x)=u(x),\ \forall x\in \Omega\e,\\
\|\Pi\e u\|_{H^1(\Omega)}\leq C\|u\|_{H^1(\Omega\e)},
\end{array}
\end{gather}
where the constant $C$ is independent of both $u$ and $\eps$.

Also we introduce the operator $\Pi_2\e:L^2(\cupl_i B_i\e)\to L^2(\Gamma)$:
\begin{gather*}
\Pi_2\e u(x)=
\begin{cases}
\langle u \rangle_{B_i\e}\sqrt{r\e},& x\in \Gamma^{\eps}_i,\\ 0,& x\in\Gamma\backslash\cupl_i
\Gamma^{\eps}_i.
\end{cases}
\end{gather*}
Using the Cauchy-Schwarz inequality and the definition \eqref{ass4} of $r\e$ we obtain 
\begin{gather}
\label{Pi2ineq} \|\Pi_2\e u\|^2_{L^2(\Gamma)}\leq r\e\suml_{i\in\mathcal{I}\e} {|\Gamma_i\e|\over|B_i\e|}
\intl_{B_i\e}|u(x)|^2\d x\leq C\suml_{i\in\mathcal{I}\e}\intl_{B_i\e}\b\e |u(x)|^2\d x\leq C_1\|u\|^2_{\mathcal{H}\e}.
\end{gather}

\begin{remark}\label{remark_q}
As we already mentioned in the introduction, the parameter $q$ characterizes the ``strength'' of coupling between $\Omega\e$ and the union of the balls $B_i\e$. If $q=\infty$ the coupling is strong:
given a family $\left\{w\e\in H^1(\Omega)\right\}_{\eps>0}$ satisfying 
\begin{gather}\label{grad}
\eta\e[w\e,w\e]\leq C
\end{gather}
one can observe that the behaviour as 
$\eps\to 
0$ of $w\e$ on $\cupl_i 
B_i\e$ is 
determined by the one in $\Omega\e$. Namely (cf. 
\eqref{u1u2}),
$$\liml_{\eps\to 0}\left\|\sqrt{r\e}\hspace{1mm} (\Pi_1\e w\e)|_\Gamma - \Pi_2\e 
w\e\right\|_{L^2(\Gamma)}=0.$$
On the other hand, if $q$ is finite then the coupling is weak: 
for an arbitrary smooth functions $w_1$, $w_2$ on $\Gamma$ 
there exists a family 
$\left\{w\e\in H^1(\Omega)\right\}_{\eps>0}$ satisfying \eqref{grad} and
$$ 
(\Pi_1 w\e)|_{\Gamma}\to w_1,\ \Pi_2\e w\e\to w_2\text{ in }L^2(\Gamma)\text{ as }\eps\to 0.$$
The function $w\e$ can be constructed, for example, by the formula \eqref{we} below.
\end{remark}

In view of \eqref{normalization1}-\eqref{Pi2ineq} the functions 
$\Pi_1\e u\e$ and $\Pi_2\e u\e$ are bounded in $H^1(\Omega)$ and $L^2(\Gamma)$, respectively, uniformly in $\eps$. Then, using the Rellich embedding theorem and the trace theorem, we conclude that there is a subsequence (for convenience, still denoted by $u_\eps$) and $u_1\in H^1(\Omega)$, $u_2\in L^2(\Gamma)$ such that
\begin{gather}\label{Pi1conv1}
\Pi_1\e u\e\underset{\eps\to 0}\rightharpoonup u_1\text{ in }H^1(\Omega),\\\label{Pi1conv2}
\Pi_1\e u\e\underset{\eps\to 0}\rightarrow u_1\text{ in }L^2(\Omega),
\\\label{Pi1conv3}
\Pi_1\e u\e\underset{\eps\to 0}\rightarrow u_1\text{ in }L^2(\Gamma).
\\\label{Pi2conv}
\Pi_2\e u\e\underset{\eps\to 0}\rightharpoonup u_2\text{ in }L^2(\Gamma).
\end{gather}
(here we use the same notation for the functions $u\e$, $u_1$ and their traces on $\Gamma$).
It is clear that $\Pi_1\e u\e=0$ on $\partial\Omega$, whence using the trace theorem we arrive at $u_1=0$ on ${\partial\Omega}$, i.e. $u_1\in H^1_0(\Omega)$.
\smallskip

Now we consider separately two cases: $u_1\not=0$ and $u_1=0$. 
\medskip

\textbf{Case 1:} $u_1\not=0$. 
We will prove prove that in this case $\lambda$ is an eigenvalue of the operator $\mathcal{A}_{q,r}$ if $r>0$ (respectively, of the operator $\mathcal{A}_{q,0}$ if $r=0$)
and $U=(u_1,r^{-1/2}u_2)$ (respectively, $U=(u_1,u_2)$) is a corresponding eigenfunction. 
\smallskip 

For an arbitrary $w\in H^1_0(\Omega)$ one has
\begin{gather}\label{int_eq}
\intl_{\Omega} a\e(x)\nabla u\e(x)\cdot \nabla w(x)\hspace{1px}  \d x =\lambda\e
\intl_{\Omega}b\e(x) u\e(x) w(x)  \d x.
\end{gather}
Our strategy of proof will be to plug into \eqref{int_eq} some specially chosen test-function $w$ depending on $\eps$ and then pass to the limit as $\eps\to 0$  in order to obtain either the equality $\mathcal{A}_{q,r} U=\lambda U$ ($r>0$) or the equality $\mathcal{A}_{q,0} U=\lambda U$ ($r=0$)  written in a weak form. 

For constructing this special test-function we introduce several additional functions. Let $\Phi:\mathbb{R}\to\mathbb{R}$ be a smooth function such that $\Phi(\rho)=1$ for $\rho\leq 1$ and $\Phi(\rho)=0$ for $\rho\geq 2$. For $i\in \mathcal{I}\e$ we denote
\begin{gather}\label{phi_i}
\phi_i^{\eps}(x)=\Phi\left({|x-x^{i,\eps}|+{\eps\over 2}-2R\e\over{\eps\over 2}-R\e}\right).
\end{gather}
It is clear that 
\begin{gather}\label{phi1}
\phi_i^{\eps}=0\text{ in }\mathbb{R}^n\setminus \cupl_{\in\I}{Y_i\e},\quad \phi_i\e=1\text{ in }\overline{D_i\e\cup B_i\e},\\\label{phi2} \mathrm{supp}(D^m\phi_i\e)\subset \overline{Y_i\e}\setminus \left(D_i\e\cup B_i\e\right)\ (m\not =0),\\ |D^m\phi_i\e|\leq {C\over \eps^{|m|}},\ m=1,2,3,\dots 
\label{phi3}
\end{gather}

By $v_i\e(x)$ we denote the following function:
\begin{gather}\label{v}
v_i\e(x)=
\begin{cases}
1,& |x-x^{i,\eps}|\leq{R\e-d\e},\\
{A\e\over  |x-x^{i,\eps}|^{n-2}}+B\e,& {R\e-d\e}< |x-x^{i,\eps}|< {R\e},\\
0,& {R\e}\leq |x-x^{i,\eps}|,
\end{cases}
\end{gather}
where
\begin{gather}\label{A}
\begin{matrix}
A\e=\ds\left({1\over (R\e-d\e)^{n-2}}-{1\over (R\e)^{n-2}}\right)^{-1},\quad B\e=-{A\e\over (R\e)^{n-2}}.
\end{matrix}
\end{gather}
Taking into account that $d\e=o(\eps)$ one gets the following asymptotics as $\eps\to 0$:
\begin{gather}\label{Ae}
A\e\sim {(R\e)^{n-1}\over (n-2)d\e}.
\end{gather}

It is easy to see that $v_i\e$, $i\in\I$, are continuous and piecewise smooth functions.
Using \eqref{ass2}-\eqref{ass4}, \eqref{Ae} one can easily obtain the following asymptotics as $\eps\to 0$:
\begin{gather}\label{v4} \intl_{Y_i\e}a\e|\nabla v_i\e|^2\d x=\a\e (A\e)^2 (n-2) \left({1\over (R\e-d\e)^{n-2}}-{1\over (R\e)^{n-2}}\right)\omega_{n-1} 
\sim q\e r\e\eps^{n-1},\\\label{v5} \intl_{Y_i\e}|v_i\e|^2b\e \d x=\b\e |B_i\e|+\mathcal{O}(d\e\eps^{n-1})\sim r\e \eps^{n-1}.
\end{gather}

Finally, taking arbitrary functions $w_1\in C_0^\infty(\Omega)$, $w_2\in C^\infty(\Gamma)$, we construct the following test function:
\begin{gather}\label{we}
w\e(x)=w_1(x)+\suml_{i\in\I}\left(w_1(x^{i,\eps})-w_1(x)\right)\phi_i\e(x)+\suml_{i\in\I}v_i\e(x)
\left({w_2(x^{i,\eps})\over\sqrt{r\e}}-w_1(x^{i,\eps})\right).
\end{gather}
It is clear that $w\e\in H^1_0(\Omega)$.

We plug $w=w\e(x)$  into \eqref{int_eq}. This gives, using \eqref{phi1}-\eqref{phi2},
\begin{multline}\label{Iomega}
\underset{\mathbf{I_1}}{\underbrace{\intl_{\Omega\e}\nabla u\e\cdot \nabla w_1 \d x}}+\underset{\mathbf{I_2}}{\underbrace{\suml_{i\in\I}\intl_{Y_i\e\setminus\overline{ D_i\e\cup B_i\e}}\nabla u\e\cdot\nabla \left(\left(w_1(x^{i,\eps})-w_1\right)\phi_i\e\right)\d x}}\\
+\underset{\mathbf{I_3}}{\underbrace{\suml_{i\in\I}\left({w_2(x^{i,\eps})\over\sqrt{r\e}}-w_1(x^{i,\eps})\right)\intl_{D_i\e}\a\e\nabla u\e\cdot\nabla v_i\e\vspace{1px} \d x}}\\=\lambda\e\Bigg(\underset{\mathbf{I_4}}{\underbrace{\intl_{\Omega\e}u\e w_1 \d x}}+\underset{\mathbf{I_5}}{\underbrace{\suml_{i\in\I}\intl_{Y_i\e\setminus\overline{ D_i\e\cup B_i\e}} u\e \left(w_1(x^{i,\eps})-w_1\right)\phi_i\e \d x}}\\+\underset{\mathbf{I_6}}{\underbrace{\suml_{i\in\I}\intl_{D_i\e}u\e w\e \d x}}+\underset{\mathbf{I_7}}{\underbrace{\suml_{i\in\I}\intl_{B_i\e}\beta\e u\e {w_2(x^{i,\eps})\over\sqrt{r\e}}\d x}}\Bigg).
\end{multline}

Let us study step-by-step the terms $\mathbf{I_j},\ \mathbf{j}=1,\dots,7$.\medskip

\noindent$\textbf{1)}$ Since $\liml_{\eps\to 0}|\Omega\setminus\Omega\e|=0$ and $\|\nabla\Pi\e u\e\|_{L^2(\Omega\setminus\Omega\e)}\leq C$, 
we obtain, using \eqref{Pi1conv1},
\begin{gather}\label{Iomega1}
\mathbf{I_1}=\intl_{\Omega}\nabla(\Pi_1\e u\e) \cdot\nabla w_1 \d x-\intl_{\Omega\setminus\Omega\e}\nabla(\Pi_1\e u\e) \cdot\nabla w_1 \d x\underset{\eps\to 0}\to \intl_{\Omega}\nabla u_1\cdot\nabla w_1 \d x.
\end{gather}\medskip

\noindent$\textbf{2)}$ Using the estimates $$\left|\nabla \left(\left(w_1(x^{i,\eps})-w_1\right)\phi_i\e\right)\right|\leq C,\quad \suml_{i\in\mathcal{I}\e}\eps^{n-1}=\suml_{i\in\mathcal{I}\e}|\Gamma_i\e|\leq |\Gamma|$$ 
(the first one follows easily from \eqref{phi1}-\eqref{phi3}) 
and taking into account that $\|\nabla u\e\|_{L^2(\Omega\e)}\leq C$
we conclude that  
\begin{gather}\label{Iomega2}
\left|\mathbf{I_2}\right|^2\leq C\|\nabla u\e\|^2_{L^2(\Omega\e)}\left|\cupl_{i\in\mathcal{I}\e} Y_i\e\right|\leq C_1\suml_{i\in\mathcal{I}\e}\eps^{n}\leq C_2\eps.
\end{gather}\medskip

\noindent$\textbf{3)}$ Integrating by parts and taking into account that $\Delta v_i\e=0$ in $D_i\e$  we get:
\begin{multline}\label{Iomega3}
\mathbf{I_3}=
\suml_{i\in\I}\left({w_2(x^{i,\eps})\over\sqrt{r\e}}-w_1(x^{i,\eps})\right)\a\e\left(\intl_{S^{\eps,+}_i}{\partial v_i\e\over\partial |x-x^{i,\eps}|}u\e \d s-\intl_{S^{\eps,-}_i}{\partial v_i\e\over\partial |x-x^{i,\eps}|}u\e \d s\right)\\=
{\a\e A\e \omega_{n-1}(n-2)}\suml_{i\in\I}\left(w_1(x^{i,\eps})-{w_2(x^{i,\eps})\over\sqrt{r\e}}\right)\left(\langle u\e\rangle_{S_i^{\eps,+}} - \langle u\e\rangle_{S_i^{\eps,-}}\right)
\end{multline}
Recall that by $\omega_{n-1}$ we denote the volume of the $(n-1)$-dimensional unit sphere, and $A\e$ is defined by \eqref{A}.

We introduce the operator $Q\e: C^1(\Gamma)\to L^2(\Gamma)$ by  
\begin{gather}\label{Qdef}
Q\e w=\begin{cases}
w(x^{i,\eps}),& x\in \Gamma^{\eps}_i,\\ 0,& x\in\Gamma\backslash\cupl_{i\in\I}
\Gamma^{\eps}_i.
\end{cases}
\end{gather}
It is easy to see that
\begin{gather}\label{Q}
\forall w\in C^1(\Gamma):\ Q\e w\underset{\eps\to 0}\to w\text{ in }L^2(\Gamma).
\end{gather}

We obtain from \eqref{Iomega3}:
\begin{multline}\label{Iomega4}
\mathbf{I_3}={\a\e A\e\omega_{n-1}(n-2)}
\suml_{i\in\I}\left(w_1(x^{i,\eps})-{w_2(x^{i,\eps})\over \sqrt{r\e}})\right)\left(\langle \Pi_1\e u\e\rangle_{\Gamma_i^{\eps}} - \langle u\e\rangle_{B_i^{\eps}}\right)+\delta(\eps)\\=
{\a\e A\e\omega_{n-1}(n-2)\over\eps^{n-1}}\intl_{\Gamma}\left(Q\e w_1-  {1\over\sqrt{r\e}}Q\e w_2\right)\left(\Pi_1\e u\e-{1\over\sqrt{r\e}}\Pi_2\e u\e\right)\d s+\delta(\eps),
\end{multline}
where the remainder $\delta(\eps)$ vanishes as $\eps\to 0$. Namely, applying \eqref{aa1} for $u:=\Pi\e_1 u\e$ and \eqref{aa2} for $u:=u\e$ and using the asymptotics
\begin{gather}\label{Ae+}
{\a\e A\e\omega_{n-1}(n-2)\over\eps^{n-1}}\sim q\e r\e\text{ as }\eps\to 0
\end{gather}
following from \eqref{Ae},  we obtain (below by $N(\eps)$ we denote the cardinality of
of the set $\mathcal{I}\e$; clearly, $N(\eps)$ satisfies \eqref{total}):
\begin{multline*}
|\delta(\eps)|^2\leq 
C\left({a\e A\e\over \sqrt{r\e}}\right)^2 N(\eps) \suml_{i\in\mathcal{I}\e}\left(\left|\langle u\e\rangle_{S_i^{\eps,+}}-\langle \Pi_1\e u\e\rangle_{\Gamma_i\e}\right|^2+
\left|\langle u\e\rangle_{S_i^{\eps,-}}-\langle u\e\rangle_{B_i\e}\right|^2\right)\\\leq C_1\eps (q\e)^2 r\e\suml_{i\in\mathcal{I}\e}\left(\|\nabla\Pi_1\e u\e\|^2_{L^2(Y_i\e)}+\|\nabla u\e\|^2_{L^2(B_i\e)}\right)
\leq C_2\eps.
\end{multline*}
Thus we obtain, using \eqref{Pi1conv3}, \eqref{Pi2conv}, \eqref{Q}-\eqref{Ae+}
\begin{multline}\label{Iomega5}
\mathbf{I_3}=
{\a\e A\e\omega_{n-1}(n-2)\over\eps^{n-1}}\intl_{\Gamma}\left(Q\e w_1-  {1\over\sqrt{r\e}}Q\e w_2\right) \left(\Pi_1\e u\e-{1\over\sqrt{r\e}}\Pi_2\e u\e\right) \d s+o(1)\\\to
\intl_{\Gamma}\left(qr u_1 w_1-q\sqrt{r}u_1 w_2-q\sqrt{r} u_2 w_1+ q u_2 w_2 \right)\d s\text{ as }\eps\to 0.
\end{multline}
\medskip

\noindent$\textbf{4)}$ Similarly to $\mathbf{I_1}$ we arrive at
\begin{gather}\label{Jomega1}\mathbf{I_4}\to \intl_{\Omega}u_1 w_1 \d x\text{ as }\eps\to 0
\end{gather}
\medskip

\noindent$\textbf{5),\ 6)}$ In view of \eqref{ass2} and since $\left| \left(w_1(x^{i,\eps})-w_1\right)\phi_i\e\right|<C\eps$ 
and $|w\e|<{C\over\sqrt{r\e}}$ in $\cupl_{i\in \I}D_i\e$ we obtain
\begin{multline}\label{Jomega2}
\left|\mathbf{I_5}+\mathbf{I_6}\right|^2 \leq C\suml_{i\in\I}\left(\eps^2|Y_i\e\setminus (D_i\e\cup B_i\e)|+{|D_i\e|\over r\e}\right)\suml_{i\in\I}\|u\e\|^2_{L^2(Y_i\e\setminus B_i\e)}\leq C_1\suml_{i\in\I}
\left(\eps^{n+2}+{d\e\eps^{n-1}\over r\e}\right)\\\leq
C_2\left(\eps^3+{\a\e\over (d\e)^2}q\e\right)\to 0\text{ as }\eps\to 0.
\end{multline}\medskip

\noindent$\textbf{7)}$ Using the equality $|B_i\e|=(R\e-d\e)^n\varkappa_n\sim R^{n}\varkappa_n\eps^n$ (recall that by $\varkappa_n$ we denote the volume of $n$-dimensional unit ball) we get:
\begin{gather}\label{Jomega3}
\mathbf{I_7}={\b\e|B_i\e|\over\eps^{n-1}}\suml_{i\in\I}\langle u\e\rangle_{B_i\e}{w_2(x^{i,\eps})\over \sqrt{r\e}}\eps^{n-1}={\b\e |B_i\e|\over\eps^{n-1} r\e}\intl_\Gamma \Pi_2\e u\e Q\e w_2 \d s\to \intl_{\Gamma}u_2 w_2 \d s.
\end{gather} 
\medskip

Combining \eqref{Iomega}-\eqref{Iomega2}, \eqref{Iomega5}-\eqref{Jomega3} we get
\begin{multline}\label{int_eq_final}
\intl_{\Omega}\nabla u_1\cdot \nabla w_1 \d x+\intl_\Gamma\left(qr u_1w_1-q\sqrt{r}u_1w_2-q\sqrt{r} u_2 w_1+ q u_2 w_2\right)\d s\\=\lambda\left(\intl_{\Omega}u_1 w_1 \d x+\intl_{\Gamma}u_2 w_2 \d s\right),\quad \forall w_1\in C_0^\infty(\Omega),w_2\in C^\infty(\Gamma).
\end{multline}
By  density arguments equality \eqref{int_eq_final} is valid for any arbitrary $w_1\in H_0^1(\Omega)$ and $w_2\in L^2(\Gamma)$.

If $r>0$ then \eqref{int_eq_final} is equivalent to the equality
\begin{gather*}
\eta_{q,r}[U,W]=\lambda(U,W)_{\mathcal{H}_r},\text{ where }U=(u_1,r^{-1/2}u_2),\ W=(w_1,r^{-1/2}w_2),
\end{gather*}
whence, obviously, 
\begin{gather*}
U\in\mathrm{dom}(\mathcal{A}_{q,r}),\ \mathcal{A}_{q,r}U=\lambda U.
\end{gather*}
Since $u_1\not=0$, $\lambda$ is therefore an eigenvalue of the operator $\mathcal{A}_{q,r}$. 

If $r=0$ then \eqref{int_eq_final} implies
\begin{gather*}
U=(u_1,u_2)\in\mathrm{dom}(\mathcal{A}_{q,0}),\ \mathcal{A}_{q,0}U=\lambda U, 
\end{gather*}
i.e. $\lambda$ is an eigenvalue of the operator $\mathcal{A}_{q,0}$.
\medskip

\textbf{Case 2:} $u_1=0$. 
We will prove that in this case
$\lambda=q$ (recall that for every $r\geq 0$ one has $q\in\sigma(\mathcal{A}_{q,r})$).
\smallskip

We express the eigenfunction $u\e$ in the form
\begin{gather}\label{u=v-g+d}
u\e=v\e-g\e+\delta\e,
\end{gather}
where
\begin{gather}\label{ve}
v\e=\suml_{i\in\I}\langle u\e\rangle_{B_i\e}v_i\e,\quad g\e=\suml_{k=1}^{k(\eps)-1}(v\e,u_k\e)_{\mathcal{H}\e}u_k\e,
\end{gather}
and $\delta\e$ s a remainder term. Here the functions $v_i\e$ are again defined by \eqref{v}-\eqref{A}. It is clear that $v\e\in H^1_0(\Omega)$ and $g\e\in\mathrm{dom}(\mathcal{A}\e)$. Also we note that
\begin{gather}\label{v_est7}
v\e=\langle u\e \rangle_{B_i\e}\text{ in }B_i\e,\quad v\e=0\text{ in }\Omega\e.
\end{gather}

At first we obtain some estimates for the eigenfunction $u\e$. 
For any $u\in H^1(Y_i\e)$ one has the estimate \cite[Lemma 4.3]{Khrab2}:
\begin{gather}\label{oldineq}
\|u\|^2_{D_i\e}\leq C\left((d\e)^2\|\nabla u\|^2_{L^2(D_i\e)}+\eps d\e\|\nabla u\|_{L^2(Y_i\e\setminus \overline{D_i\e\cup B_i\e})}^2+\eps^{-1}d\e\|u\|^2_{L^2(Y_i\e\setminus \overline{D_i\e\cup B_i\e})}\right),
\end{gather}
Recall that $d\e=o(\eps)$ and $(d\e)^2=o(\a\e)$. Using this and \eqref{normalization2}, we obtain from \eqref{oldineq}:
\begin{gather}\label{u_inD}
\liml_{\eps\to 0}\suml_{i\in\I}\|u\e\|^2_{L^2(D_i\e)}= 0.
\end{gather}

Finally, using the Poincar\'{e} inequality and \eqref{normalization2}, we obtain:
\begin{gather}\label{u_inB}
\suml_{i\in\I}\|u\e-\langle u\e \rangle_{B_i\e}\|^2_{L^2(B_i\e)}=\mathcal{O}(\eps^2)\text{ as }\eps\to 0.
\end{gather}
Since $\eps\b\e=\mathcal{O}(1)$ (in view of \eqref{ass4}),  \eqref{u_inB} implies
\begin{gather}\label{u_inB+}
\suml_{i\in\I}\beta\e\|u\e-\langle u\e \rangle_{B_i\e}\|^2_{L^2(B_i\e)}=\mathcal{O}(\eps)\text{ as }\eps\to 0.
\end{gather}

Using the fact that 
\begin{gather}\label{u=0}
\|u\e\|_{L^2(\Omega\e)}\leq  \|\Pi_1\e u\e\|_{L^2(\Omega)}\underset{\eps\to 0}\to \|u_1\|_{L^2(\Omega)}=0 
\end{gather}
and \eqref{normalization1}, \eqref{u_inD}, \eqref{u_inB+} 
we obtain
\begin{gather*}
1=\|u\e\|^2_{\mathcal{H}\e}=\suml_{i\in\I}\left|\langle u\e \rangle_{B_i\e}\right|^2 |B_i\e|\b\e+o(1)=\eps\beta\e{|B_i\e|\over\eps^{n}} \suml_{i\in\I}\left|\langle u\e \rangle_{B_i\e}\right|^2\eps^{n-1}+o(1)\quad (\eps\to 0),
\end{gather*}
whence, taking into account that $\eps\beta\e{|B_i\e|\over\eps^{n}}\sim
{\eps\b\e (R\e)^n \chi_n\over \eps^n}
\sim r\e$ (here we use $d\e=o(\eps)$), 
\begin{gather}\label{u_L2}
r\e\suml_{i\in\I}\left|\langle u\e \rangle_{B_i\e}\right|^2\eps^{n-1}\sim 1\text{ as }\eps\to 0.
\end{gather}

Using \eqref{ass2}-\eqref{ass4}, \eqref{v4}, \eqref{v5} and taking \eqref{u_L2} into account we obtain the following estimates:
\begin{gather}\label{v_est1}
\eta\e[v\e,v\e]=\suml_{i\in\I}\left(\intl_{Y_i\e}a\e(x)|\nabla v_i\e|^2 \d x\right) \left|\langle u\e \rangle_{B_i\e}\right|^2\sim q\text{ as }\eps\to 0,\\\label{v_est2}
\|v\e\|^2_{\mathcal{H}\e}=\suml_{i\in\I}\left(\intl_{Y_i\e}b\e(x)|v_i\e|^2 \d x\right) \left|\langle u\e \rangle_{B_i\e}\right|^2\sim 1\text{ as }\eps\to 0,\\\label{v_est6}
\suml_{i\in\I}\|v\e\|^2_{L^2(D_i\e)}\leq |D_i\e|\suml_{i\in\I}\left|\langle u\e \rangle_{B_i\e}\right|^2\leq C_1 q\e {(d\e)^2\over \a\e}  r\e\suml_{i\in\I}\left|\langle u\e \rangle_{B_i\e}\right|^2\eps^{n-1} \to 0\text{ as }\eps\to 0.
\end{gather}

From the Bessel inequality and the orthogonality of eigenfunctions (namely, $(u\e,u_k\e)_{\mathcal{H}\e}=0$ provided $k<k(\eps)$) we obtain the following estimates for the function $g\e$:
\begin{gather}\label{g1}
\|g\e\|^2_{\mathcal{H}\e}=\suml_{k=1}^{k(\eps)-1}\left|(v\e,u_k\e)_{\mathcal{H}\e}\right|^2=
\suml_{k=1}^{k(\eps)-1}\left|(v\e-u\e,u_k\e)_{\mathcal{H}\e}\right|^2\leq \|v\e-u\e\|^2_{\mathcal{H}\e},\\\label{g2}
\eta\e [g\e,g\e]=\suml_{k=1}^{k(\eps)-1} \lambda_k\e\left|(v\e, u_k\e)_{\mathcal{H}\e}\right|^2=
\suml_{k=1}^{k(\eps)-1}\lambda_k\e\left|(v\e-u\e,u_k\e)_{\mathcal{H}\e}\right|^2\leq \lambda\e\|v\e-u\e\|^2_{\mathcal{H}\e}.
\end{gather}
In view of  \eqref{v_est7}, \eqref{u_inD}, \eqref{u_inB+}, \eqref{u=0}, \eqref{v_est6} and the fact that $u_1=0$ one has
\begin{gather}
\label{u-v}
\|u\e-v\e\|^2_{\mathcal{H}\e}=\|u\e\|^2_{L^2(\Omega\e)}+\suml_{i\in\I}\|u\e-v\e\|^2_{L^2(D_i\e)}+
\suml_{i\in\I}\beta\e\|u\e-\langle u\e \rangle_{B_i\e}\|^2_{L^2(B_i\e)}\to 0\text{ as }\eps\to 0
\end{gather}
and therefore, by virtue of \eqref{g1}-\eqref{g2}, 
\begin{gather}\label{g_main}
\|g\e\|^2_{\mathcal{H}\e}+\eta\e [g\e,g\e]\to 0\text{ as }\eps\to 0.
\end{gather}

Now let us estimate the remainder $\delta\e$. 
We denote $\tilde v\e=v\e-g\e$. Since $\tilde v\e\in\{u_1\e,\dots,u_{k(\eps)-1}\e\}^\perp$, one has by the well-known variational characterization of eigenvalues (see, e.g., \cite{Reed}) 
\begin{gather*}
\eta\e[u\e,u\e]=\lambda\e=\min\left\{{\eta\e[u,u]\over \|u\|^2_{\mathcal{H}\e}};\ (u,u_k\e)_{\mathcal{H}\e}=0\text{ as }k=1,\dots,k(\eps)-1\right\}
\leq {\eta\e[\tilde v\e,\tilde v\e]\over \|\tilde v\e\|^2_{\mathcal{H}\e}}.
\end{gather*}
or equivalently, using $u\e=\tilde v\e+\delta\e$,
\begin{gather}\label{minmax}
\eta\e[\delta\e,\delta\e]\leq -2\eta\e[\tilde v\e,\delta\e]+{\eta\e[\tilde v\e,\tilde v\e]}\left(\|\tilde v\e\|^{-2}_{\mathcal{H}\e}-1\right).
\end{gather}

In view of \eqref{v_est1}, \eqref{v_est2}, \eqref{g_main} the second term on the right-hand-side of \eqref{minmax} tends to zero as $\eps\to 0$:
\begin{gather}\label{minmax1}
{\eta\e[\tilde v\e,\tilde v\e]}\left(\|\tilde v\e\|^{-2}_{\mathcal{H}\e}-1\right)\to 0\text{ as }\eps\to 0.
\end{gather}

Now, let us estimate the first term. One has
\begin{gather}\label{minmax2}
\eta\e[\tilde v\e,\delta\e]=\eta\e[v\e,u\e-v\e]+\eta\e[v\e,g\e]-\eta\e[g\e,\delta\e].
\end{gather}
Integrating by parts and using \eqref{v}, \eqref{A} and \eqref{v_est7} we get:
\begin{multline}\label{delta1} 
\eta\e[v\e,u\e-v\e]=\suml_{i\in \I}\intl_{D_i\e}\a\e \nabla v_i\e\cdot\nabla (u\e-v\e)\d x\\=
\a\e \langle u\e \rangle_{B_i^{\eps}} \suml_{i\in \I} \left(\intl_{S^{\eps,+}_i}{\partial v_i\e\over\partial |x-x^{i,\eps}|}u\e \d s-\intl_{S^{\eps,-}_i}{\partial v_i\e\over\partial |x-x^{i,\eps}|}(u\e-\langle u\e \rangle_{B_i\e}) \d s\right)\\
={\a\e A\e \omega_{n-1}(n-2)}\suml_{i\in \I} \langle u\e \rangle_{B_i^{\eps}} \left(-\langle u\e \rangle_{S_i^{\eps,+}}+\langle u\e \rangle_{S_i^{\eps,-}}-\langle u\e \rangle_{B_i^{\eps}}\right).
\end{multline}
Then, using the Cauchy-Schwarz inequality,  \eqref{Pi1conv3}, \eqref{Ae+}, \eqref{u_L2}, Lemma \ref{lemma_est} and the fact that $u_1=0$, we obtain  from \eqref{delta1}:
\begin{multline}\label{delta2}
\left|\eta\e[v\e,u\e-v\e]\right|^2\leq C\left({\a\e A\e}\right)^2\left\{ \suml_{i\in \I}\left|\langle u\e \rangle_{B_i^{\eps}}\right|^2 \right\}
\left\{\suml_{i\in\I}\left(\left|\langle u\e \rangle_{S_i^{\eps,+}}\right|^2+\left|\langle u\e \rangle_{S_i^{\eps,-}}-\langle u\e \rangle_{B_i^{\eps}}\right|^2\right)\right\}
\\\leq 
C_1 (q\e)^2 r\e \eps^{n-1}
\suml_{i\in\I}\left(\left|\langle \Pi_1\e u\e \rangle_{\Gamma_i^{\eps}}\right|^2+\left|\langle u\e \rangle_{S_i^{\eps,+}}-\langle \Pi_1\e  u\e\rangle_{\Gamma_i^{\eps}}\right|^2+\left|\langle  u\e \rangle_{S_i^{\eps,-}}-\langle u\e \rangle_{B_i^{\eps}}\right|^2\right)\\\leq C_2\left(\|\Pi_1\e u\e\|_{L^2(\Gamma)}^2+\eps\|\nabla \Pi_1\e u\e\|_{L^2(\cupl_iY_i\e)}^2+\eps\|\nabla u\e\|_{L^2(\cupl_i B_i\e)}^2\right)\to 0\text{ as }\eps\to 0.
\end{multline}
Further, in view of \eqref{v_est1}, \eqref{g_main},
\begin{gather}\label{minmax4}
\liml_{\eps\to 0}\eta\e[v\e,g\e]=0.
\end{gather}
And finally, using \eqref{normalization2}, \eqref{v_est1}, \eqref{g_main}, we obtain:
\begin{gather}\label{minmax5}
\left|\eta\e[g\e,\delta\e]\right|\leq \left|\eta\e[g\e,u\e]\right|+\left|\eta\e[g\e,v\e]\right|+\left|\eta\e[g\e,g\e]\right|\to 0\text{ as }\eps\to 0.
\end{gather}
It follows from \eqref{minmax2}, \eqref{delta2}-\eqref{minmax5} that
\begin{gather}\label{delta3}
\liml_{\eps\to 0}\eta\e[\tilde v\e,\delta\e]=0.
\end{gather}
Combining \eqref{minmax}, \eqref{minmax1}, \eqref{delta3} we conclude that
\begin{gather}\label{delta_H1}
\liml_{\eps\to 0}\eta\e[\delta\e,\delta\e]=0.
\end{gather}

Finally, using \eqref{normalization2}, \eqref{u=v-g+d}, \eqref{v_est1}, \eqref{g_main},  \eqref{delta_H1}, we obtain:
\begin{gather*}
\lambda=\liml_{\eps\to 0}\lambda\e=\liml_{\eps\to 0}\eta\e[u\e,u\e]=\liml_{\eps\to 0}\eta\e[v\e,v\e]=q.
\end{gather*}

Property \eqref{ha} is completely proved. 

\subsubsection{Proof of property \eqref{hb}} Let $\lambda\in \sigma(\mathcal{A}_{q,r})$ if $r>0$ (respectively,  $\lambda\in \sigma({\mathcal{A}_{q,0}})$ if $r=0$). We have to prove that
\begin{gather}\label{Hb1}
\text{there exists a family }\left\{\lambda\e\in\sigma(\mathcal{A}\e)\right\}_\eps:\ \lambda\e\to\lambda\text{ as }\eps\to 0
\end{gather}
or, equivalently,
$$\forall\delta>0\ \exists\eps(\delta)\text{ such that }\forall\eps<\eps(\delta):  (\lambda-\delta,\lambda+\delta)\cap\sigma(\mathcal{A}^{\eps})\not=\varnothing.$$

For proving this indirectly we assume the opposite. Then a positive number
$\delta$ and a
subsequence (for convenience still indexed by $\eps$) exist such that
\begin{gather}
\label{dist} (\lambda-\delta,\lambda+\delta)\cap\sigma(\mathcal{A}^{\eps})=\varnothing.
\end{gather}

Since $\lambda\in\sigma(\mathcal{A}_{q,r})$ (respectively, $\lambda\in\sigma(\mathcal{A}_{q,0})$) there exists
$F=(f_1,f_2) \in L^2(\Omega)\oplus L^2(\Gamma)$, such that
\begin{gather}\label{notinim}
F\notin
\mathrm{range}(\mathcal{A}_{q,r} -\lambda I)\
\text{(respectively, }F\notin\mathrm{range}(\mathcal{A}_{q,0} -\lambda I)\text{)}.
\end{gather}

We introduce the function $f\e\in\mathcal{H}\e$ by 
\begin{gather*}
f\e(x)=\begin{cases} f_1(x),&x\in\Omega\e,\\
0,& x\in\cupl_{i\in\I}D_i\e,\\
{1\over\sqrt{r\e}}{\langle \mathbf{f}_2 \rangle_{\Gamma_i\e}},&x\in
B_i\e.
\end{cases}
\end{gather*}
Here $\mathbf{f}_2(x)=\sqrt{r} f_2(x)$ if $r>0$ (respectively, $\mathbf{f}_2(x)=f_2(x)$ if $r=0$).

Taking \eqref{ass4}  into account  we obtain:
\begin{gather*}
\|f\e\|^2_{\mathcal{H}\e}=\|f_1\|^2_{L^2(\Omega\e)}+{1\over r\e}\suml_{i\in\I}\beta\e|B_i\e|\left|\langle \mathbf{f}_2 \rangle_{\Gamma_i\e}\right|^2\leq \|f_1\|^2_{L^2(\Omega)}+\|\mathbf{f}_2\|^2_{L^2(\Gamma)}\leq C\|F\|^2_{L^2(\Omega)\oplus L^2(\Gamma)}.
\end{gather*}

By virtue of \eqref{dist} $\lambda$ is in the resolvent set of $\mathcal{A}\e$. Hence  there exists a unique $u\e\in\mathrm{dom}(\mathcal{A}\e)$ satisfying
$$\mathcal{A}\e u\e-\lambda u\e=f\e$$
and the following estimates are valid:
\begin{gather}\label{uL2+}
\|u\e\|_{\mathcal{H}\e}\leq \delta^{-1}\|f\e\|_{\mathcal{H}\e}\leq C_1,\\\label{uH1+}
\eta\e[u\e,u\e]= \lambda\|u\e\|^2_{\mathcal{H}\e}+(f\e,u\e)_{\mathcal{H}\e}\leq C_2.
\end{gather}
Estimates \eqref{uL2+}-\eqref{uH1+} imply the existence of a subsequence (again indexed by $\eps$)  and $u_1\in H_0^1(\Omega)$, $u_2\in L^2(\Gamma)$ satisfying
\eqref{Pi1conv1}-\eqref{Pi2conv}.
 
For an arbitrary $w\in H^1(\Omega)$ one has the equality:
\begin{gather}\label{int_eq_f}
\intl_{\Omega}a\e \nabla u\e \cdot \nabla w  \d x -\lambda\intl_{\Omega}b\e  u\e w \d x=\intl_{\Omega}b\e f\e w   \d x.
\end{gather}
We plug into \eqref{int_eq_f} the function $w=w\e(x)$ defined by  \eqref{we} and pass to the limit as $\eps\to 0$. 
Using the same arguments as in the proof of property \eqref{ha} we conclude that that $u_1,u_2$ satisfy
\begin{multline}\label{int_eq_f_final}
\intl_{\Omega}\nabla u_1\cdot \nabla w_1 \d x+\intl_{\Gamma}\left(qr u_1 w_1-q\sqrt{r}u_2 w_1-q\sqrt{r}u_1 w_2+q u_2 w_2\right)\d s-\lambda\left(\intl_{\Omega}u_1 w_1 \d x+\intl_{\Gamma}u_2 w_2 \d s\right)\\=\intl_{\Omega}f_1  w_1 \d x+\intl_{\Gamma}\mathbf{f}_2   w_2 \d s,
\end{multline}
for an arbitrary $w_1\in C_0^\infty(\Omega)$, $w_2\in C^\infty(\Gamma)$ (and by  density arguments for any arbitrary $w_1\in H^1_0(\Omega)$, $w_2\in L^2(\Gamma)$).
It follows from \eqref{int_eq_f_final} that 
\begin{gather*}
\begin{array}{llll}
\text{if }r>0\text{ then }&U=(u_1,r^{-1/2}u_2)\in\mathrm{dom}(\mathcal{A}_{q,r})&\text{ and }&\mathcal{A}_{q,r} U-\lambda U=F,\\
\text{if }r=0\text{ then }& U=(u_1,u_2)\in\mathrm{dom}(\mathcal{A}_{q,0})&\text{ and }&\mathcal{A}_{q,0}  U-\lambda U=F.
\end{array}
\end{gather*}
This contradicts to \eqref{notinim}. Thus there is $\lambda\e\in\sigma(\mathcal{A}\e)$ such that $\liml_{\eps\to 0}\lambda\e=\lambda$. Property \eqref{hb} is proved 
which finishes the proof of Theorem \ref{th1} for the case $q<\infty$.

\subsection{\label{subsec33}Spectrum of operator $\mathcal{A}_{q,r}$ ($0<q<\infty$, $r>0$)}

This subsection is devoted to the proof of Lemma \ref{lmAqr}.
First we study the discrete spectrum of the operator $\mathcal{A}_{q,r}$. Let $\lambda\not=q$ be an eigenvalue of $\mathcal{A}_{q,r}$ corresponding to the eigenfunction $U=(u_1,u_2)\in\mathcal{H}_r$. This means that
\begin{multline}
\label{disc1}
\intl_{\Omega}\nabla u_1\cdot\nabla \overline{v_1} \d x+qr\intl_{\Gamma}(u_1-u_2)\overline{(v_1-v_2)}\d s\\=\lambda\left(\intl_{\Omega}u_1 \overline{v_1}\d x+r\intl_{\Gamma}u_2 \overline{v_2}\d s\right),\quad \forall v=(v_1,v_2)\in\mathrm{dom}({\eta}_{q,r}).
\end{multline}
One can easily get from \eqref{disc1} (taking first the test-function $(0,v_2)$ and then $(v_1,0)$):
\begin{multline}\label{disc2}
U=(u_1,u_2)\in \mathrm{dom}({\eta}_{q,r})\text{ satisfies \eqref{disc1} }\Longleftrightarrow\
u_2={qu_1\over q-\lambda}\text{ and }\\
\intl_{\Omega}\nabla u_1\cdot\nabla \overline{v_1} \d x-{\lambda qr\over q-\lambda}\intl_{\Gamma}u_1\overline{v_1} \d s=\lambda\intl_{\Omega}u_1\overline{v_1} \d x,\ \forall v_1\in H_0^1(\Omega).
\end{multline}

Let $\mu\in\mathbb{R}$. By $\eta^\mu$ we denote the sesquilinear form in $L^2(\Omega)$ defined as follows:
$$\eta^\mu[u,v]=\intl_{\Omega}\nabla u\cdot\nabla \overline{v} \d x-\mu \intl_\Gamma u\overline{v} \d s,\quad \mathrm{dom}(\eta^\mu)=H^1_0(\Omega).$$
We denote by $\mathcal{A}^\mu$ the operator associated with this form.
Formally the eigenvalue problem $\mathcal{A}^\mu u=\lambda u$ can be written as
\begin{gather}\label{mu_problem}
\begin{cases}
-\Delta u=\lambda u&\text{in }\Omega\setminus\Gamma,\\ \left[u\right]=0 & \text{on } \Gamma,\\
\left[{\partial u\over\partial  n}\right]-\mu u=0&\text{on }\Gamma,\\
u=0&\text{on }\partial\Omega.
\end{cases}
\end{gather}
The spectrum of $\mathcal{A}^\mu$ is purely discrete. We denote by $$\lambda_1(\mu)\leq\lambda_2(\mu)\leq\dots\leq\lambda_k(\mu)\leq\dots \underset{k\to\infty}\to\infty$$ the
sequence of eigenvalues of $\mathcal{A}^\mu$ repeated according to their multiplicity. By $\{u_k(\mu)\}_{k=1}^\infty$ we denote the corresponding sequence of eigenfunctions satisfying $(u_k(\mu),u_l(\mu))_{L^2(\Omega)}=\delta_{kl}$.

We denote by $\sigma_{p}(\mathcal{A}_{q,r})$ the set of eigenvalues of the operator $\mathcal{A}_{q,r}$. It follows easily from \eqref{disc2} that
\begin{gather}\label{lambdamu}
\sigma_{p}(\mathcal{A}_{q,r})\setminus\{q\}=\left\{\lambda\in\mathbb{R}:\ \lambda=\lambda_k(\mu)={q\mu\over qr+\mu}\text{ for some }\mu\in\mathbb{R}\text{ and some }k\in\mathbb{N}
\right\}
\end{gather}

We also introduce the operator  $\mathcal{A}^D$  acting in $L^2(\Omega)$ being associated with the form  $\eta^D$, which is defined as follows:
$$\mathrm{dom}(\eta^D)=\{u\in H_0^1(\Omega):\ u=0\text{ on }\Gamma\},\quad \eta^D[u,v]=\intl_{\Omega}\nabla u\cdot\nabla v\d x.$$
We denote by  $\{\lambda_k^D\}_{k=1}^\infty$ the sequence of eigenvalues of $\mathcal{A}^D$ written in increasing order and with account of their multiplicity.

Below we establish some properties of the spectrum of the operator $\mathcal{A}^\mu$.
\begin{proposition}\label{lm1}
One has for each fixed $k\in\mathbb{N}$:
\begin{flalign}\label{lm1a}
I.\quad &\text{the function }\lambda_k(\cdot):\mathbb{R}\to \mathbb{R}\text{ is continuous and monotonically decreasing},
\\\label{lm1b}
II.\quad&\lambda_k(\mu)\to\lambda_k^D\text{ as }\mu\to -\infty,
\\\label{lm1c}
III.\quad&\lambda_k(\mu)\to -\infty\text{ as }\mu\to \infty.
\end{flalign}
\end{proposition}

{
\begin{proof}
\textit{I.} One has due to the min-max principle (see, e.g., \cite{Davies}):
\begin{gather}\label{minmax_p}
\lambda_k(\mu)=\minl_{L\in \mathcal{L}_k}
\maxl_{u\in L}{\eta^\mu[u,u]\over \|u\|^2_{L^2(\Omega)}},\ k=1,2,3\dots
\end{gather}
where $\mathcal{L}_k$ is the set of all $k$-dimensional subspaces of $H^1_0(\Omega)$. Then the monotonicity follows easily from \eqref{minmax_p} and the monotonicity (for fixed $u$) of the function $\mu\mapsto \eta^\mu[u,u]$.

Now, let us prove continuity. Let $[\mu_0,\mu_1]\subset\mathbb{R}$  be an arbitrary compact interval. We choose some $\tau>0$ such that
\begin{gather}\label{cont1}
\tau\mu_1\leq{1\over 2}
\end{gather}
(if $\mu_1<0$ we can choose an arbitrary $\tau>0$). By a standard trace inequality there exists $C_\tau>0$ such that
\begin{gather}\label{cont2}
\forall u\in H^1(\Omega):\ \|u\|^2_{L^2(\Gamma)}\leq \tau
\|\nabla u\|_{L^2(\Omega)}^2+C_\tau\|u\|^2_{L^2(\Omega)}.
\end{gather}

Now, let $\mu,\tilde\mu\in[\mu_0,\mu_1]$, $\mu\leq\tilde\mu$. We denote for abbreviation:
\begin{gather}\label{cont3}
\a:=1+{\tau(\tilde\mu-\mu)\over 1-\tau\tilde\mu},\quad\beta:={C_\tau(\tilde\mu-\mu)\over 1-\tau\tilde\mu}.
\end{gather}
In view of \eqref{cont1} $\a$ and $\b$ are positive.

For each $u\in H^1_0(\Omega)\setminus\{0\}$ we obtain, using \eqref{cont2},
\begin{multline*}
(1-\a)\|\nabla u\|^2_{L^2(\Omega)}+(\a\tilde\mu-\mu)\|u\|^2_{L^2(\Gamma)}\\\leq
\left(1-\a+(\a\tilde\mu-\mu)\tau\right)\|\nabla u\|^2_{L^2(\Omega)}+\left(\a\tilde\mu-\mu\right)C_\tau\|u\|^2_{L^2(\Omega)}\overset{\eqref{cont3}}=\beta \|u\|^2_{L^2(\Omega)}
\end{multline*}
or, equivalently,
\begin{gather}\label{cont4}
{\eta^\mu[u,u]\over \|u\|^2_{L^2(\Omega)}}\leq\a {\eta^{\tilde\mu}[u,u]\over \|u\|^2_{L^2(\Omega)}}+\beta.
\end{gather}

It follows from \eqref{minmax_p} and \eqref{cont4} that for each fixed $k\in\mathbb{N}$
\begin{gather*}
\lambda_k(\mu)\leq \a\lambda_k(\tilde\mu)+\beta,
\end{gather*}
Using \eqref{cont1}, \eqref{cont3} and the monotonicity of $\lambda_k(\cdot)$, we obtain:
\begin{multline}\label{cont5}
0\leq \lambda_k(\mu)-\lambda_k(\tilde\mu)\leq{\tau\lambda_k(\tilde\mu)+C_\tau\over 1-\tau\tilde\mu}(\tilde\mu-\mu)\\\leq {\tau\lambda_k(\mu_0)+C_\tau\over 1-\tau\mu_1}(\tilde\mu-\mu)\leq 2(\tau\lambda_k(\mu_0)+C_\tau)(\tilde\mu-\mu)
\end{multline}
which implies the desired continuity on the interval $[\mu_0,\mu_1]$. Since this interval was chosen arbitrarily we obtain the continuity on the whole axis. \medskip

\textit{II.} 
It is easy to see that 
$$\mathrm{dom}(\eta^D)=\left\{u\in H_0^1(\Omega):\ \sup_{\mu\leq 0} \eta^\mu[u,u]\right\}$$
and
$$\eta^D[u,v]=\eta^\mu[u,v]\text{ for }u,v\in \mathrm{dom}(\eta^D).$$
Then, using the monotonicity (for fixed $u$) of the function $\mu\mapsto \eta^\mu[u,u]$ and
Theorem 3.1 from \cite{Simon}, we conclude that for each $f\in L^2(\Omega)$
\begin{gather}\label{strong}
(\mathcal{A}^\mu+I)^{-1}f \to (\mathcal{A}^D+I)^{-1}f\text{ in }L^2(\Omega)\text{ as }\mu\to -\infty.
\end{gather}
Moreover, since the operators $(\mathcal{A}^\mu+I)^{-1}$ and $(\mathcal{A}^D+I)^{-1}$ are compact and $(\mathcal{A}^{\mu_1}+ I)^{-1}\geq (\mathcal{A}^{\mu_2}+ I)^{-1}\geq 0$
for $\mu_1\geq\mu_2$, by virtue of  \cite[Theorem VIII-3.5]{K66} 
\eqref{strong} can be improved to the norm convergence  
\begin{gather*}
\|(\mathcal{A}^\mu+ I)^{-1}-(\mathcal{A}^D+ I)^{-1}\|\to 0\text{ as 
}\mu\to -\infty,
\end{gather*}
implying  the convergence of eigenvalues \eqref{lm1b}.
\medskip

\textit{III.} Let $m\in\mathbb{N}$. Let $B_j$, $j=1,\dots,m$ be the open balls with centers at some points $z_j\in \Gamma$ and with  radius $b$. It is supposed $z_j$ and $b$ are such that
\begin{gather}\label{plus1}
B_j\subset\Omega,\ j=1,\dots,m\text{ and }B_i\cap B_j=\varnothing,\ i\not= j.
\end{gather}

Let $v(x)$ be an arbitrary smooth function such that $v(x)>0$ for $|x|< b$ and $v(x)=0$ for $|x|\geq b$. We denote $v_j(x)=v(x-z_j)$. Since $\supp(v_j)\subset B_j$, we have $v_j\in H^1_0(\Omega)$. We denote
$$L=\mathrm{span}\{v_1,\dots,v_m\}.$$

It is clear that $\mathrm{dim}(L)=m$, and thus using \eqref{minmax_p} we get
\begin{gather}\label{plus2}
\lambda_m(\mu)\leq\maxl_{u\in L}{\eta^{\mu}[u,u]\over \|u\|^2_{L^2(\Omega)}}.
\end{gather}
Let $0\not= \tilde  u\in L$ maximize the quotient on the right-hand-side of \eqref{plus2}. It can be represented in the form $\tilde u=\suml_{j=1}^m \a_j v_j$, where $\a_j\in\mathbb{R}$, $\a:=\suml_{j=1}^m \a_j^2>0$, and hence
$$\|\tilde u\|_{L^2(\Omega)}^2=\a\|v\|^2_{L^2(\{|x|<b\})},\quad
\|\tilde u\|_{L^2(\Gamma)}^2=\a\|v\|^2_{L^2(\{|x|<b,x_n=0\})},\quad \|\nabla \tilde u\|_{L^2(\Omega)}^2=\a \|\nabla v\|^2_{L^2(\{|x|<b\})}.$$
Then, taking into account that $\|v\|^2_{L^2(\{|x|<b\})}>0$, $\|v\|^2_{L^2(\{|x|<b,x_n=0\})}>0$, we get
\begin{gather}\label{plus3}
\lambda_m(\mu)= {\|\nabla \tilde u\|^2_{L^2(\Omega)}-\mu \|\tilde u\|^2_{L^2(\Gamma)}\over \|\tilde u\|^2_{L^2(\Omega)}}\leq A-\mu B,\text{ where } B>0.
\end{gather}
Then \eqref{lm1c} follows directly from \eqref{plus3}.
\end{proof}
}

Now, with Proposition \ref{lm1} we can easily establish the properties of the set on the right-hand-side of \eqref{lambdamu}. 
We denote by ${\mathcal{C}}$ the curve $${\mathcal{C}}=\left\{(\lambda,\mu)\in\mathbb{R}^2: \lambda={q\mu \over qr+\mu}\right\}.$$ It consists of two branches ${\mathcal{C}}_{\pm}=\left\{(\lambda,\mu)\in{\mathcal{C}}: \pm(\mu+qr)> 0\right\}$.  
We also introduce the curves ${\mathcal{C}}_k=\{(\lambda,\mu)\in\mathbb{R}^2:\ \lambda=\lambda_k(\mu)\}$, $k\in\mathbb{N}$.

It follows easily from \eqref{lm1a}-\eqref{lm1c} that 
\begin{itemize}
\item For each $k\in\mathbb{N}$ the curve ${\mathcal{C}}_k$ intersects
the curve ${\mathcal{C}}_{+}$ exactly in one point. We denote the corresponding value of $\lambda$ by $\lambda_k^+$. 

\item  We denote by $k_0$ the smallest integer satisfying $\lambda_{k_0}^D\leq q$ and $\lambda_{k_0+1}^D> q$. Then for each $k\in\mathbb{N}$ the curve ${\mathcal{C}}_{k_0+k}$ intersects 
the curve ${\mathcal{C}}_{-}$ exactly in one point. We denote the corresponding value of $\lambda$ by $\lambda_k^-$. For $k\leq k_0$ the curve ${\mathcal{C}}_k$ has no intersections with the curve ${\mathcal{C}}_-$.

\item \eqref{distr} holds true.
\end{itemize}
Thus, with \eqref{lambdamu}, we conclude that
\begin{gather}\label{lambdamu+}
\sigma_{p}(\mathcal{A}_{q,r})\setminus\{q\}=\{\lambda_k^-,k=1,2,3...\}\cup\{\lambda_k^+,k=1,2,3...\}.
\end{gather}

Since $\lambda_k^+\to q$ as $k\to\infty$, we have  $q\in\sigma_{ess}(\mathcal{A}_{q,r})$. 
To complete the proof of Lemma \ref{lmAqr} we have to show that
if $\lambda\not =q$ then $\lambda\not\in\sigma_{ess}(\mathcal{A}_{q,r})$

We denote by $\mathcal{B}$ the following operator acting in $\mathcal{H}_r$:
$$\mathrm{dom}(\mathcal{B})=H^1_0(\Omega)\oplus L^2(\Gamma),\  \mathcal{B}U=(0,qu_1|_\Gamma),\text{ where }U=(u_1,u_2).$$
In view of the embedding and trace theorems the operator $\mathcal{B}(\mathcal{A}_{q,r}+ I)^{-1}$ is compact, that is $\mathcal{B}$ is an $\mathcal{A}_{q,r}$-compact operator. Thus (see, e.g, \cite[Theorem 1.9]{ELZ}) the operator $$\widetilde{\mathcal{A}}:=\mathcal{A}_{q,r}+\mathcal{B}$$
with $\mathrm{dom}(\widetilde{\mathcal{A}}):=\mathrm{dom}(\mathcal{A}_{q,r})$ 
is closed and 
$$\sigma_{ess}(\mathcal{A}_{q,r})=\sigma_{ess}(\widetilde{\mathcal{A}}),$$
where the essential spectrum of non-self-adjoint operator $\widetilde{\mathcal{A}}$ is understood in the following sense\footnote{There are several ways how to define the essential spectrum for non-self-adjoint operators.
All the definitions can be found, for example, in \cite{ELZ} (for self-adjoint operators they are equivalent). One of the possible ways is to define it by \eqref{ess_def}. The advantage of this definition is twofold: the so-defined essential spectrum can be characterized via singular sequences (see \eqref{Weyl}) and it is stable under  relatively compact perturbations.}: 
\begin{gather}\label{ess_def}
\sigma_{ess}(\widetilde{\mathcal{A}}):=\mathbb{C}\setminus
\left\{\lambda:\ \mathrm{range}(\widetilde{\mathcal{A}}-\lambda I)\text{ is closed}\text{ and }\mathrm{dim}(\mathrm{ker}(\widetilde{\mathcal{A}}-\lambda I))<\infty\right\}.
\end{gather}
Moreover in view of  \cite[Theorem 1.6]{ELZ}  $\lambda$ belongs to $\sigma_{ess}(\widetilde{\mathcal{A}})$
iff 
\begin{gather}\label{Weyl}
\exists \left\{U^k\in\mathrm{dom}(\widetilde{\mathcal{A}})\right\}_k\text{ such that }\|U^k\|_{\mathcal{H}_r}=1,\ U^k \rightharpoonup 0\text{ in }\mathcal{H}_r,\ \widetilde{\mathcal{A}}U^k-\lambda U^k\to 0\text{ as }{k\to \infty}.
\end{gather}

Suppose that $\lambda\not =q$ and let us prove that $\lambda\not\in\sigma_{ess}(\widetilde{\mathcal{A}})$.
We assume the opposite. Then there exists a sequence $U^k=(u_1^k,u_2^k)\in\mathrm{dom}(\widetilde{\mathcal{A}})$, $k\in\mathbb{N}$ satisfying \eqref{Weyl}.
Consequently,  for an arbitrary sequence $V^k=(v^k_1,v^k_2)\in H^1_0(\Omega)\oplus L^2(\Gamma)$ 
satisfying $\|V^k\|_{\mathcal{H}_r}\leq C$ one has
$$(\widetilde{\mathcal{A}}U^k-\lambda U^k,V^k)_{\mathcal{H}_r}\to 0\text{ as }k\to\infty$$
or, equivalently, 
\begin{multline}\label{Weyl+}
(\nabla u^k_1,\nabla v^k_1)_{L^2(\Omega)}+qr (u_1^k,v_1^k)_{L^2(\Gamma)}-qr (u^k_2,v^k_1)_{L^2(\Gamma)}+qr (u_2^k,v_2^k)_{L^2(\Gamma)} \\
-\lambda (u_1^k,v_1^k)_{L^2(\Omega)}
-\lambda r (u_2^k,v_2^k)_{L^2(\Gamma)}\to 0\text{ as }k\to \infty.
\end{multline}
Plugging into  \eqref{Weyl+} $V^k:=(0,u_2^k)$ and taking into account that $\lambda\not=q$ and $r>0$ we obtain that
\begin{gather}\label{u2to0} 
u_2^k\to 0\text{ as }k\to\infty. 
\end{gather}
Then, taking in  \eqref{Weyl+} $V^k:=(v,0)$, where $v$ is an arbitrary function belonging to $H^1_0(\Omega)$, we obtain:
\begin{gather}\label{Weyl++}
\langle u^k_1,v\rangle -qr (u^k_2,v)_{L^2(\Gamma)}
-\lambda (u_1^k,v)_{L^2(\Omega)}\to 0\text{ as }k\to \infty,
\end{gather}
where $\langle u,v\rangle:= (\nabla u,\nabla v)_{L^2(\Omega)}+qr (u,v)_{L^2(\Gamma)}$.
In view of the Friedrichs and trace inequalities the norm $\|u\|:=\langle u,u\rangle$ is equivalent to the standard Sobolev norm in $H^1_0(\Omega)$. Then, using \eqref{u2to0} and \eqref{Weyl++} and taking into account that $\|u_1^k\|_{L^2(\Omega)}\leq C$, we conclude that
\begin{gather*}\label{Weyl+++}
|\langle u^k_1,v\rangle|\leq C,\ \forall v\in H^1_0(\Omega).
\end{gather*}
Then by the Banach-Steinhaus Theorem 
$$\langle u^k_1,u_1^k\rangle \leq C,$$
i.e. $\{u_1^k\}_k$ is bounded in $H^1(\Omega)$. Thus by the Rellich embedding theorem the sequence $u_1^k$ is compact in $L^2(\Omega)$, which, together with \eqref{u2to0}, contradicts to 
$\|U^k\|_{\mathcal{H}_r}=1,\ U^k \rightharpoonup 0\text{ in }\mathcal{H}_r$.

Lemma \ref{lmAqr} is proved.

\subsection{\label{subsec34}Proof of Theorem \ref{th1}: the case $q=\infty$}

Let $\lambda\e\in\sigma(\mathcal{A}\e)$ and  $\lambda\e\to\lambda$ as $\eps\to 0$. We have to prove that   
\begin{gather*}
\lambda\in\sigma(\mathcal{A}_{\infty,r})\text{ if }r>0,\text{\quad and\quad  }\lambda\in\sigma(\mathcal{A}_{\infty,0})\text{ if }r=0.
\end{gather*}
 
Again by $u\e$ we denote the eigenfunction corresponding to $\lambda\e$ and satisfying \eqref{normalization1}-\eqref{normalization2}. In the same way as in the case $q<\infty$ we conclude that there exists $(u_1,u_2)\in H^1_0(\Omega)\oplus L^2(\Gamma)$ such that \eqref{Pi1conv1}-\eqref{Pi1conv3} hold.

For an arbitrary $w\in H^1_0(\Omega)$ one has the equality \eqref{int_eq}. Let $w_0$ be an arbitrary function from $C^\infty_0(\Omega)$. We plug into \eqref{int_eq} the function $w=w\e$ defined by 
\begin{gather}\label{we1} 
w\e(x)=w_0(x)+\suml_{i\in\I}\left(w_0(x^{i,\eps})-w_0(x)\right)\phi_i\e(x),
\end{gather}
where the function $\phi_i\e$ is defined by \eqref{phi_i}.

Taking into account 
that the integral
$\intl_{\cupl_i (D_i\e\cup B_i\e)}\a\e\nabla u\e\cdot\nabla w\e \d x$ is equal to zero (since $w\e=w_0(x^{i,\eps})=\mathrm{const}.$ in $D_i\e\cup B_i\e$) 
we pass to the limit in \eqref{int_eq} 
and via the same arguments as in the case $q<\infty$ we get:
\begin{gather}\label{inttilde}
\intl_{\Omega}\nabla u_1\nabla w_0 \d x=\lambda \intl_\Omega u_1 w_0 \d x+\lambda\intl_{\Gamma}\sqrt{r} u_2 w_0 \d s.
\end{gather}

Let us prove that $\sqrt{r} u_1|_\Gamma=u_2$. Let $w$ be an arbitrary function from $C^1(\Gamma)$, the operator $Q\e$ be defined by  \eqref{Qdef}.
One has, using \eqref{ass4}, \eqref{Pi1conv3}, \eqref{Pi2conv} and \eqref{Q}:
\begin{multline}\label{weak1}
\intl_\Gamma (\sqrt{r}u_1-u_2)w \d s=\liml_{\eps\to 0}\intl_\Gamma (\sqrt{r\e}\Pi_1\e u\e-\Pi_2\e u\e)(Q\e w )\d s\\=
\liml_{\eps\to 0}\suml_{i\in\I}\intl_{\Gamma_i\e}  (\sqrt{r\e}\Pi_1\e u\e-\Pi_2\e u\e)(Q\e w) \d s\\=\liml_{\eps\to 0}\suml_{i\in\I}\sqrt{r\e}\left(\langle \Pi_1\e u\e\rangle_{\Gamma_i\e}-\langle u\e\rangle_{B_i\e}\right)w(x^{i,\eps})|\Gamma_i\e|.
\end{multline}

Using \eqref{aa1}, \eqref{aa2}, \eqref{ineq_pm}, the inequality $\ds\suml_{i\in \I}\eps^{n-1}\le C$ and taking into account \eqref{ass3}, \eqref{ass4}, \eqref{ab}, \eqref{eta} we obtain from \eqref{weak1}:
\begin{multline*}
\left|\intl_\Gamma (\sqrt{r}u_1-u_2)w \d s\right|^2\\\leq C\liml_{\eps\to 0}\left(\sqrt{r\e}\suml_{i\in\I}\eps^{n-1}\left|\langle \Pi_1\e u\e\rangle_{\Gamma_i\e}-\langle u\e\rangle_{B_i\e}\right|\right)^2\leq C\liml_{\eps\to 0}r\e\left(\suml_{i\in\I}\eps^{n-1}\left|\langle \Pi_1\e  u\e\rangle_{\Gamma_i\e}-\langle u\e\rangle_{B_i\e}\right|^2\right) \suml_{i\in\I}\eps^{n-1} \\
\leq C_1\liml_{\eps\to 0}\suml_{i\in\I}r\e\eps^{n-1}\left(\left|\langle \Pi_1\e u\e\rangle_{\Gamma_i\e}-\langle u\e\rangle_{S_i^{\eps,+}}\right|^2+\left|\langle u\e\rangle_{S_i^{\eps,+}}-\langle u\e\rangle_{S_i^{\eps,-}}\right|^2+\left|\langle u\e\rangle_{S_i^{\eps,-}}-\langle u\e\rangle_{B_i^{\eps}}\right|^2\right)\\\leq
C_2\liml_{\eps\to 0}r\e\left(\eps\|\nabla \Pi_1\e  u\e\|^2_{L^2(\cupl_i Y_i\e)}+d\e\|\nabla u\e\|^2_{L^2(\cupl_i D_i\e)}+\eps\|\nabla u\e\|^2_{L^2(\cupl_i B_i\e)}\right)\leq C_3\liml_{\eps\to 0}\left(\eps r\e+{1\over q\e}\right)\eta\e[u\e,u\e]= 0.
\end{multline*}
Thus $\intl_\Gamma (\sqrt{r}u_1-u_2)w \d s=0$ for all $w\in C^1(\Gamma)$, whence 
\begin{gather}\label{u1u2}
\sqrt{r}u_1|_\Gamma=u_2.
\end{gather}

It follows from \eqref{inttilde}, \eqref{u1u2} that 
\begin{gather}\label{llll}
\begin{array}{llll}
\text{if }r>0&\text{then}&U=(u_1,u_1|_\Gamma)\in\mathrm{dom}(\mathcal{A}_{\infty,r}),& \mathcal{A}_{\infty,r}U=\lambda U,\\
\text{if }r=0&\text{then}&u_1\in\mathrm{dom}(\mathcal{A}_{\infty,0}),&\mathcal{A}_{\infty,0}u_1=\lambda u_1. 
\end{array}
\end{gather}

Finally we prove that $u_1\not=0$. One has
\begin{gather}\label{1}
1=\|u\e\|^2_{\mathcal{H}\e}=\|u\e\|^2_{L^2(\Omega\e)}+\suml_{i\in\I}\|u\e\|^2_{L^2(D_i\e)}+
\suml_{i\in\I}\beta\e\|u\e\|^2_{L^2(B_i\e)}.
\end{gather}
Therefore, using \eqref{Pi1conv2} and taking in mind that $u\e=\Pi_1\e u\e$ in $\Omega\e$, $\liml_{\eps\to 0}|\Omega\setminus\Omega\e|=0$:
\begin{gather}\label{1.1}
\liml_{\eps\to 0}\|u\e\|_{L^2(\Omega\e)}\leq \liml_{\eps\to 0}\left(\|u_1\|_{L^2(\Omega\e)}+\|u\e-u_1\|_{L^2(\Omega\e)}\right)=\|u_1\|_{L^2(\Omega)}.
\end{gather}
In the same way as in the case $q<\infty$ (see \eqref{u_inD}) we get:
\begin{gather}\label{1.2}
\liml_{\eps\to 0}\suml_{i\in\I}\|u\e\|^2_{L^2(D_i\e)}=0
\end{gather}
(recall that the proof of \eqref{u_inD} relies on inequality \eqref{oldineq} and the condition $(d\e)^2=o(\a\e)$). 
Finally, using the Poincar\'{e} inequality, \eqref{aa1}, \eqref{aa2}, \eqref{ineq_pm} and the Cauchy-Schwarz inequality, we get
\begin{multline}\label{1.3}
\suml_{i\in\I}\beta\e\|u\e\|^2_{L^2(B_i\e)}\leq C\suml_{i\in\I}\b\e\left(\|u\e-\langle u\e\rangle_{B_i^{\eps}}\|^2_{L^2(B_i\e)}+\eps^{n}\left|\langle u\e\rangle_{B_i^{\eps}}-\langle u\e\rangle_{S_i^{\eps,-}}\right|^2\right.\\\left.+\eps^{n}\left|\langle u\e\rangle_{S_i^{\eps,-}}-\langle u\e\rangle_{S_i^{\eps,+}}\right|^2+\eps^{n}\left|\langle u\e\rangle_{S_i^{\eps,+}}-\langle \Pi_1\e u\e\rangle_{\Gamma_i^{\eps}}\right|^2+\eps^{n}\left|\langle \Pi_1\e u\e\rangle_{\Gamma_i^{\eps}}\right|^2\right)\\\leq
C_1\left(\eps r\e\|\nabla u\e\|^2_{L^2(\cupl_{i}B_i\e)}+{1\over q\e}  \a\e\|\nabla u\e\|^2_{L^2(\cupl_{i}D_i\e)}+\eps r\e\|\nabla \Pi_1\e u\e\|^2_{L^2(\cupl_{i}Y_i\e)}\right)+C_1 r\e\|\Pi_1\e u\e\|^2_{L^2(\Gamma)}
\end{multline}
Passing to the limit in \eqref{1.3} and taking \eqref{normalization2} into account  we obtain
\begin{gather}\label{1.4}
\liml_{\eps\to 0}\suml_{i\in\I}\beta\e\|u\e\|^2_{L^2(B_i\e)}\leq C\|u_1\|^2_{L^2(\Gamma)}.
\end{gather}

It follows from \eqref{1}-\eqref{1.2}, \eqref{1.4} that $u_1\not = 0$. Therefore in view of \eqref{llll} $\lambda$ is an eigenvalue of $\mathcal{A}_{\infty,r}$ as $r>0$, and
and eigenvalue of $\mathcal{A}_{\infty,0}$ as $r=0$.

Property \eqref{hb} of the Hausdorff convergence is proved in the same way as in the case $q<\infty$ (using the test-function ${w}\e(x)$ defined before by \eqref{we1} instead of the one defined by \eqref{we}). 
This completes the proof of Theorem \ref{th1} for the case $q=\infty$.

\section{\label{sec4}Spectrum of a waveguide}

In this section we consider the unbounded waveguide type domain $\Omega\subset\mathbb{R}^2$:
$$\Omega=\mathbb{R}\times \left(d_-,d_+\right),\quad  d_-<0<d_+.$$
In this case $$\Gamma=\{x=(x_1,x_2):\ x_2=0\}.$$
We again suppose that conditions \eqref{ass2}-\eqref{ass4} hold. In what follows we consider the case $q>0$, $r>0$ only.  

In the same way as before we introduce the Hilbert spaces $\mathcal{H}\e$ and $\mathcal{H}_r$, and the operators $\mathcal{A}\e$ and $\mathcal{A}_{q,r}$. Furthermore, for brevity we will use the notations $\mathcal{H}$ and $\mathcal{A}$ instead of $\mathcal{H}_r$, $\mathcal{A}_{q,r}$, correspondingly.  

In order to state the result we need to introduce some additional notations.
For fixed $\mu\in\mathbb{R}$  we denote by $\a(\mu)$ the smallest eigenvalue of the problem
\begin{gather*}
-u''=\lambda u\text{ in }\left(d_-,d_+\right)\setminus\{0\},\\
u(d_-)=u(d_+)=0,\\
u(-0)=u(+0),\quad u'(-0)-u'(+0)=\mu u(0).
\end{gather*}
Similarly to the proof of Lemma \ref{lmAqr} we conclude that the function  $\mu\mapsto \a(\mu)$ is continuous, monotonically decreasing and moreover $\alpha(\mu)\underset{\mu\to -\infty}\to \minl\left\{\left({\pi\over d_-}\right)^2,\left({\pi\over d_+}\right)^2\right\}$\ \footnote{This is the smallest eigenvalue of the problem $-u''=\lambda u\text{ on }(d_-,d_+)\setminus\{0\},\ u(d_-)=u(0)=u(d_+).$} and $\alpha(\mu)\underset{\mu\to\infty}\to -\infty$. 
Using these properties one can easily conclude that there exists one and only one point $\mu_1\in (-qr,\infty)$ satisfying
\begin{gather*}
\a(\mu_1)={q\mu_1 \over qr+\mu_1},
\end{gather*}
and if $q<\min\left\{{\pi^2\over d_-^2},{\pi^2\over d_+^2}\right\}$ then there exists one and only one point $\mu_2\in (-\infty,-qr)$
satisfying
\begin{gather*}
\a(\mu_2)={q\mu_2 \over qr+\mu_2},
\end{gather*}
moreover
$$0<\a(\mu_1)<q<\a(\mu_2).$$

Now we are able to present the main result of this section.
\begin{theorem}\label{th3}
Let $l\subset \mathbb{R}$ be an arbitrary compact interval.
Then the set $\sigma(\mathcal{A}\e)\cap l$ converges in the Hausdorff 
sense as $\eps\to 0$ to the set $\sigma(\mathcal{A})\cap l$.

The spectrum of the operator $\mathcal{A}$ has the following form:
\begin{gather}\label{aqa}
\sigma(\mathcal{A})=\mathcal{D}:=
\begin{cases}
[\a(\mu_1),q]\cup[\a(\mu_2),\infty)&\text{ if }q<\min\left\{{\pi^2\over d_-^2},{\pi^2\over d_+^2}\right\},\\
[\a(\mu_1),\infty)&\text{ otherwise.}
\end{cases}
\end{gather}

\end{theorem}

\begin{proof}First let us prove \eqref{aqa}. 

For $L>0$ we denote $$\Omega^L=\left\{x=(x_1,x_2)\in \mathbb{R}^2:\ x_1\in (-L,L),\ x_2\in \left(d_-,d_+\right)\right\},\quad \Gamma^L=\Gamma\cap\Omega^L.$$

By $\mathcal{H}^L$ we denote the Hilbert space of functions from $L^2(\Omega^L)\oplus L^2(\Gamma^L)$ and the scalar product defined by \eqref{H_qr} with $\Omega^L$ and $\Gamma^L$ instead of $\Omega$ and $\Gamma$.

We denote by $\eta^{L,\#}$  the sesquilinear form in $\mathcal{H}^L$
which is defined by \eqref{eta_qr} (with $\Omega^L$ and $\Gamma^L$ instead of $\Omega$ and $\Gamma$) and the domain $$\mathrm{dom}(\eta^{L,\#})=\left\{(u_1,u_2)\in H^1(\Omega^L)\oplus L^2(\Gamma^L):\quad u_1(-L,\cdot) = u_1(L,\cdot),\quad u_1(\cdot,d_-)=u_1(\cdot,d_+)=0\right\}.$$
We define the operator $\mathcal{A}^{L,\#}$  acting in $\mathcal{H}^L$ being associated with this form. 

In the same way as in Lemma \ref{lmAqr} (or via direct calculations) we conclude that
the spectrum of $\mathcal{A}^{L,\#}$ consists of the point $q$ (the only point of the essential spectrum) and two sequences of eigenvalues
$\left\{\lambda_k^{L,\#,\pm}\right\}_{k\in\mathbb{N}}$ 
satisfying \eqref{distr} (with $\lambda_k^{L,\#,\pm}$ instead of $\lambda_k^{\pm}$).
Furthermore, it is easy to see that
$\liml_{L\to\infty}\lambda_1^{L,\#,-}=\a(\mu_1)$, while $\liml_{L\to\infty}\lambda_1^{L,\#,-} = \a(\mu_2)$ if $q<\min\left\{{\pi^2\over d_-^2},{\pi^2\over d_+^2}\right\}$, otherwise $\liml_{L\to\infty}\lambda_1^{L,\#,-} = q$. Moreover $\lambda_k^{L,\#,+}$ (respectively, $\lambda_k^{L,\#,-}$) are distributed more and more dense on the interval $[\lambda_1^{L,\#,-},q]$ (respectively, on the ray $[\lambda_1^{L,\#,-},\infty)$) as $L$ increases, namely
one has the equality
\begin{gather}\label{cupAL}
\cupl_{L=1}^\infty\sigma(\mathcal{A}^{L,\#})=\mathcal{D}.
\end{gather}

Let $\lambda$ be an eigenvalue of $\mathcal{A}^{L,\#}$, $U$ be the corresponding eigenfunction normalized by $\|U\|_{\mathcal{H}^L}=1$. We extend $U$ to the whole $\Omega$ by periodicity and set
$$U_N(x)={1\over \sqrt{N}}U(x)\Phi\left(N^{-1}|x_1|\right),$$
where $\Phi:\mathbb{R}\to \mathbb{R}$ is a smooth function such that $\Phi(r)=1$ as $r\leq 1$ and $\Phi(r)=0$ as $r\geq 2$. It is easy to show that
\begin{gather*} U_N\in \mathrm{dom}(\mathcal{A}),\quad
\|\mathcal{A} U_N-\lambda U_N\|_{\mathcal{H}}\to 0\text{ as }{N\to\infty},\\
0<C_1\leq\|U_N\|_{\mathcal{H}}\leq C_2
\end{gather*}
(the constants $C_1,C_2$ are independent of $N$) and therefore (see, e.g., \cite{Davies}) $\lambda\in \sigma(\mathcal{A})$. Thus we have proved that
\begin{gather*}
\forall L>0:\  \sigma(\mathcal{A}^{L,\#})\subset\sigma(\mathcal{A}),
\end{gather*}
whence, in view of \eqref{cupAL}, $$\mathcal{D}\subset \sigma(\mathcal{A}).$$

Now, we prove the reverse enclosure. Let $\lambda\in \mathbb{R}\setminus\mathcal{D}$. We have to prove that $\lambda$ belongs to the resolvent set of $\mathcal{A}$, i.e. for every $  F\in\mathcal{H}$ there is $U\in\mathrm{dom}(\mathcal{A})$ such that $\mathcal{A}U-\lambda U=F$.

We denote by $\eta^{L}$ the sesquilinear form which is defined
by \eqref{eta_qr}  (with $\Omega^L$ and $\Gamma^L$ instead of $\Omega$ and $\Gamma$) and the domain $\mathrm{dom}(\eta^{L})=H^1_0(\Omega^L)\oplus L^2(\Gamma^L)$. Let $\mathcal{A}^L$ be the operator acting in $\mathcal{H}^L$ which associated with this form.

One can easily calculate that $\sigma(\mathcal{A}^L)\subset\mathcal{D}$ for every $L>0$ and therefore, since $\lambda\notin \mathcal{D}$, for every $F^L\in\mathcal{H}^L$ there is $U^L=(u^L_1,u^L_2)\in\mathrm{dom}(\mathcal{A}^L)$ such that $\mathcal{A}^LU^L-\lambda U^L=F^L$.

Let us fix $F=(f_1,f_2)\in\mathcal{H}$ and set $F^L:=(f_1|_{\Omega^L},f_2|_{\Gamma^L})$. One has
\begin{gather}\label{UL2}
\|U^L\|_{\mathcal{H}_{\Omega^L}}\leq \mathrm{dist}(\lambda,\mathcal{D})\|F^L\|_{\mathcal{H}^L}= \mathrm{dist}(\lambda,\mathcal{D})\|F\|_{\mathcal{H}}\leq C,
\end{gather}
and as a consequence
\begin{gather}\label{UH1}
\|\nabla u^L_1\|_{\mathcal{H}^L}\leq C.
\end{gather}

We extend $u^L_1$ (respectively, $u^L_2$) by $0$ to $\Omega$ (respectively, $\Gamma$) using the same notations for the extended functions.
Obviously $u^L_1\in H^1_0(\Omega)$, $u^L_2\in L^2(\Gamma)$. It follows from \eqref{UL2}-\eqref{UH1} that there exists a subsequence (still indexed by $L$) and $u_1\in H^1(\Omega)$ and $u_2\in L^2(\Gamma)$ such that
\begin{gather}\label{u12L}
u^L_1\rightharpoonup u_1\text{ in }H^1(\Omega),\quad u^L_2
\rightharpoonup u_2\text{ in }L^2(\Gamma)\text{ as }L\to\infty.
\end{gather}

Let $(w_1,w_2)\in C^\infty_0(\Omega)\oplus L^2(\Gamma)$. When $L$ is large enough then $\supp(w_1)\subset\Omega^L$ and therefore one can write:
\begin{multline}\label{int_eq_f_final+}
\intl_{\Omega}\nabla u^L_1\cdot \nabla w_1 \d x+\intl_\Gamma qr (u^L_1-u^L_2)(w_1-w_2)\d s-\lambda\left(\intl_{\Omega}u^L_1 w_1 \d x+r\intl_{\Gamma}u^L_2 w_2 \d s\right)\\=\intl_{\Omega}f_1 w_1 \d x+r\intl_{\Gamma}f_2   w_2 \d s.
\end{multline}
Using \eqref{u12L} we pass to the limit in \eqref{int_eq_f_final+} and obtain that $U=(u_1,u_2)$ also satisfies 
\eqref{int_eq_f_final+}, i.e.
$$\mathcal{A}U-\lambda U=F.$$
Thus  $\lambda$ belongs to the resolvent set of $\mathcal{A}$. 
This finishes the proof of \eqref{aqa}. 
\medskip

Let us now turn to the proof of the Hausdorff convergence. Recall that we have to show the fulfilment of properties \eqref{ha} and \eqref{hb}.

The proof of   property \eqref{hb} of the Hausdorff convergence repeats word-by-word the proof in case of a bounded domain $\Omega$. Therefore we focus of the proof of property \eqref{ha}: let $\lambda\e\in \sigma(\mathcal{A}\e)$, $\liml_{\eps\to 0}\lambda\e=\lambda$, and we have to prove that $\lambda\in \sigma(\mathcal{A})$.

For the sake of clarity we suppose that the shells are centered at the points 
\begin{gather}\label{shells}
\tilde x^{i,\eps}=(i\eps+{1\over 2}\eps,0),
\end{gather}
i.e. we shift the shells along $\Gamma$ by $\eps/2$. Obviously, this shift does not change the spectrum of $\mathcal{A}\e$.

We denote
$${\Om}=\left\{x=(x_1,x_2)\in \mathbb{R}^2:\ x_1\in (0,1),\ x_2\in \left(d_-,d_+\right)\right\},\quad \Ga=\Gamma\cap\Om.$$
It is clear that  $$a\e(x_1+1,x_2)=a\e(x_1,x_2),\ b\e(x_1+1,x_2)=b\e(x_1,x_2)\ \text{ provided }\eps^{-1}\in\mathbb{N},$$ 
i.e. $\mathcal{A}\e$ is periodic with respect to the cell $\Om$ provided $\eps^{-1}\in\mathbb{N}$. Moreover, in view of the shift \eqref{shells}, one has for $\eps^{-1}\in\mathbb{N}$:
$$\cupl_{Y_i\e\subset\Om}\Gamma_i\e=\cupl_{i=1}^{N(\eps)}\Gamma_i\e=\widetilde\Gamma\text{, where }N(\eps)=\eps^{-1}.$$ 
Further we will study the subsequence $\lambda^{\eps_k}$, where $\eps_k={k}^{-1}$, $k=1,2,3\dots$. For convenience we will use the notation $\eps$ keeping in mind $\eps_k$. 

It is well-known from   Floquet-Bloch theory (see, e.g., \cite{Brown,Eastham,Kuchment}) that the spectrum of $\mathcal{A}\e$ can be expressed as a union of spectra of certain operators on the period cell. By $\widetilde{\mathcal{H}}\e$ we denote the space of functions from $L^2(\Om)$ and the scalar product defined by \eqref{He} with $\Om$ instead of $\Omega$. Let $\phi\in [0,2\pi)$. In the space $\widetilde{\mathcal{H}}\e$ we consider the sesquilinear form $\widetilde{\eta}^{\phi,\eps}$ defined by \eqref{eta} (with $\Om$ instead of $\Omega$) and the domain
$$\mathrm{dom}(\widetilde\eta^{\phi,\eps})=\left\{u\in H^1(\Om):\quad u(0,\cdot){=}e^{i\phi}u(1,\cdot),\quad u(\cdot,d_-)=u(\cdot,d_+)=0\right\}.$$
By $\widetilde{\mathcal{A}}^{\phi,\eps}$ we denote the operator associated with this form.
The spectrum of $\widetilde{\mathcal{A}}^{\phi,\eps}$ is purely discrete. We denote by $\left\{\widetilde\lambda_k^{\phi,\eps}\right\}_{k\in\mathbb{N}}$ the sequence of its eigenvalues 
written in the ascending order and with account to their multiplicity.
One has the following representation:
\begin{gather}\label{floque}
\sigma(\mathcal{A}\e)=\cupl_{k=1}^\infty I_k\e,\text{ where }I_k\e=\cupl_{\phi\in[0,2\pi)}\left\{\widetilde\lambda_{k}^{\phi,\eps}\right\}.
\end{gather}
The sets $I_k\e$ are compact intervals.

We also introduce the operator $\widetilde{\mathcal{A}}^{\phi}$ as the operator acting in 
$\widetilde{\mathcal{H}}:=L^2(\Om)\oplus L^2(\Ga)$ equipped with a scalar product defined by \eqref{H_qr} (with $\Om,\Ga$ instead of $\Omega,\Gamma$), being associated with the sesquilinear form $\widetilde \eta^\phi$ which is defined by \eqref{eta_qr} (with $\Om,\Ga$ instead of $\Omega,\Gamma$) and domain $\mathrm{dom}(\widetilde \eta^\phi)=\mathrm{dom}(\eta^{\phi,\eps})\oplus L^2(\Ga)$.

In the same way as in Lemma \ref{lmAqr} we conclude that
\begin{gather*}
\label{spectrum} \sigma(\widetilde{\mathcal{A}}^{\phi})=\{q\}\cup
\{\widetilde\lambda_k^{\phi,-},k=1,2,3...\}\cup\{\widetilde\lambda_k^{\phi,+},k=1,2,3...\},
\end{gather*}
the points $\widetilde\lambda_k^{\phi,\pm}, k=1,2,3...$ belong to the discrete
spectrum, $q$ is the only point of the essential spectrum and  
\begin{gather}\label{spec1}
\a(\mu_1)\leq\widetilde\lambda_1^{\phi,+}\leq \widetilde\lambda_2^{\phi,+}\leq ...\leq\widetilde\lambda_k^{\phi,+}\leq\dots\underset{k\to\infty}\to
q< \widetilde\lambda_1^{\phi,-}\leq
\widetilde 
\lambda_2^{\phi,-}\leq ...\leq\widetilde\lambda_k^{\phi,-}\leq\dots\underset{k\to\infty}\to \infty.\end{gather}
Moreover if $q<\min\left\{{\pi^2\over d_-^2},{\pi^2\over d_+^2}\right\}$ then
\begin{gather}\label{spec2}
\a(\mu_2)\leq   \widetilde\lambda_1^{\phi,-}.
\end{gather}

In view of \eqref{floque} there exists $\phi\e\in [0,2\pi)$ such that $\lambda\e\in\sigma(\widetilde{\mathcal{A}}^{\phi\e,\eps})$. We extract a subsequence (still indexed by $\eps$) such that
\begin{gather}\label{phi_to_phi}
\phi\e\to\phi\in [0,2\pi]\text{ as }\eps\to 0.
\end{gather}

Let $u\e$ be the eigenfunction of $\widetilde{\mathcal{A}}^{\phi\e,\eps}$ 
corresponding to $\lambda\e$ and normalized by the condition $\|u\e\|_{\widetilde{\mathcal{H}}\e}=1$.

In the same way as in the proof of Theorem \ref{th1} we conclude that there exists a subsequence (still indexed by $\eps$), $u_1\in H^1(\Om)$ and $u_2\in L^2(\Ga)$ such that
\begin{gather}\label{conv_phi}
\Pi_1\e u\e\rightharpoonup u_1\text{ in }H^1(\Om),\
\Pi_1\e u\e\rightarrow u_1\text{ in }L^2(\Om),\ \Pi_1\e u\e\rightarrow u_1\text{ in }L^2(\Ga),\
\Pi_2\e u\e\rightharpoonup u_2\text{ in }L^2(\Ga).
\end{gather}
(the operators $\Pi_1\e$ and $\Pi_2\e$ were defined in Subsection \ref{subsec32}). In particular, it follows from \eqref{phi_to_phi}-\eqref{conv_phi} that $U=(u_1,u_2)\in \mathrm{dom}(\widetilde\eta^{\phi})$.

If $u_1=0$ then $\lambda=q$, the proof repeats word-by-word the proof of this fact in Theorem \ref{th1}.

Now, let $u_1\not= 0$. For an arbitrary $w\in \mathrm{dom}(\eta^{\phi\e,\eps})$ we have
\begin{gather}\label{in_eq_phi}
\intl_\Om a\e\nabla u\e\cdot \nabla \overline{w} \d x=\lambda\e\intl_\Om b\e u\e \overline{w} \d x.
\end{gather}

Let $w_1,\ w_2$ be an arbitrary functions from $C^{\infty}(\Om)$ and $C^\infty(\Ga)$, respectively, moreover satisfying $$w_1(0,\cdot){=}e^{i\phi}w_1(1,\cdot),\quad w_1(\cdot,d_-)=w_1(\cdot,d_+)=0.$$
Using these functions we construct the function $w\e$ by the formula \eqref{we}. It is clear that $w\e\in \mathrm{dom}(\widetilde\eta^{\phi,\eps})$.
Finally we set
$$\hat w\e(x)=w\e(x)\left((e^{i(\phi\e-\phi)}-1)(1-x_1)+1\right),\ x=(x_1,x_2).$$
It is easy to see that $\hat w\e\in \mathrm{dom}(\widetilde\eta^{\phi\e,\eps})$ and $\widetilde\eta^{\phi\e,\eps}[\hat w\e-w\e,\hat w\e-w\e]+\|\hat w\e-w\e\|^2_{\widetilde{\mathcal{H}}\e}\underset{\eps\to 0}\to 0$.

Plugging $w=\hat w\e(x)$ into \eqref{in_eq_phi} we obtain
\begin{gather}\label{in_eq_psi+}
\intl_{\Om}a\e\nabla u\e\cdot\nabla \overline{w\e} \d x +\delta(\eps)=\intl_\Om b\e u\e \overline{w\e} \d x,
\end{gather}
where the remainder $\delta(\eps)$ vanishes as $\eps\to 0$:
\begin{gather*}
|\delta(\eps)|^2\leq 2\lambda\e \eta^{\phi\e,\eps}[\hat w\e-w\e,\hat w\e-w\e]+2\|\hat w\e-w\e\|^2_{\widetilde{\mathcal{H}}\e}  \underset{\eps\to 0}\to 0 .
\end{gather*}

Then passing to the limit $\eps\to 0$ in \eqref{in_eq_psi+} in the same way as in the proof of Theorem \ref{th1} we obtain:
\begin{gather}\label{last}
\intl_{\Om}\nabla u_1\cdot\nabla  \overline{w_1} \d x+\intl_\Ga \left(qr u_1\overline{w_1} +q\sqrt{r}u_1\overline{w_2}+q\sqrt{r}u_2\overline{w_1}+ q u_2\overline{w_2}\right)   \d s=\lambda\left(\intl_{\Om}u_1 \overline{w_1} \d x+\intl_{\Ga}u_2 \overline{w_2} \d s\right).
\end{gather}
Using density arguments we conclude that \eqref{last} holds for any arbitrary $(w_1,w_2)\in\mathrm{dom}(\widetilde\eta^{\varphi})$, whence, evidently,
$$\widetilde{\mathcal{A}}^{\phi}U=\lambda U.$$
Since $u_1\not=0$ we obtain $\lambda\in\sigma(\widetilde{\mathcal{A}}^{\phi})$. Then in view of \eqref{spec1}-\eqref{spec2} $\lambda\in \mathcal{D}$.

Theorem \eqref{th2} is proved.                                                                                                                                                                                                                                                                                                                                                                                                                                                                                              \end{proof}

\section*{Acknowledgements}We gratefully acknowledge financial support by the Deutsche
Forschungsgemeinschaft (DFG) through RTG 1294 and CRC 1173.

\end{document}